\documentclass[leqno,11pt]{article}

\RequirePackage[amsthm,amsmath]{imsart}

\usepackage[mathscr]{eucal}
\usepackage{graphicx}
\usepackage{latexsym}
\usepackage{amssymb}
\usepackage{psfrag}
\usepackage{epsfig}
\usepackage{MnSymbol}
\usepackage{enumerate}
\usepackage{multirow}
\DeclareMathAlphabet{\mathpzc}{OT1}{pzc}{m}{it}

\include{psbox}

\usepackage{amsfonts}
\usepackage{graphicx}
\usepackage{amsfonts}
\usepackage{color}
\usepackage[usenames,dvipsnames]{xcolor}
\usepackage[textsize=small]{todonotes}
\usepackage{caption,subcaption}
\usepackage{soul}

\usepackage[numbers]{natbib}

\begin{document}

\renewcommand*{\bibfont}{\small}
	\newtheorem{proposition}{Proposition}[section]
	\newtheorem{theorem}[proposition]{Theorem}
	\newtheorem{corollary}[proposition]{Corollary}
	\newtheorem{lemma}[proposition]{Lemma}
	\newtheorem{conjecture}[proposition]{Conjecture}
	\newtheorem{question}[proposition]{Question}
	\newtheorem{definition}[proposition]{Definition}
	\newtheorem{comment}[proposition]{Comment}
	\newtheorem{algorithm}[proposition]{Algorithm}
	\newtheorem{assumption}[proposition]{Assumption}
	\newtheorem{condition}[proposition]{Condition}
	\numberwithin{equation}{section}
	\numberwithin{proposition}{section}

\captionsetup[table]{format=plain,labelformat=simple,labelsep=period}

\newcommand{\skp}{\vspace{\baselineskip}}
\newcommand{\noi}{\noindent}
\newcommand{\osc}{\mbox{osc}}
\newcommand{\lfl}{\lfloor}
\newcommand{\rfl}{\rfloor}

\theoremstyle{remark}
\newtheorem{example}{\bf Example}[section]
\newtheorem{remark}{\bf Remark}[section]

\newcommand{\img}{\imath}
\newcommand{\iy}{\infty}
\newcommand{\eps}{\varepsilon}
\newcommand{\del}{\delta}
\newcommand{\Rk}{\mathbb{R}^k}
\newcommand{\RR}{\mathbb{R}}
\newcommand{\spa}{\vspace{.2in}}
\newcommand{\V}{\mathcal{V}}
\newcommand{\E}{\mathbb{E}}
\newcommand{\I}{\mathbb{I}}
\newcommand{\PP}{\mathbb{P}}
\newcommand{\sgn}{\mbox{sgn}}
\newcommand{\ti}{\tilde}

\newcommand{\QQ}{\mathbb{Q}}

\newcommand{\XX}{\mathbb{X}}
\newcommand{\XXz}{\mathbb{X}^0}

\newcommand{\lan}{\langle}
\newcommand{\ran}{\rangle}
\newcommand{\llan}{\llangle}
\newcommand{\rran}{\rrangle}
\newcommand{\lf}{\lfloor}
\newcommand{\rf}{\rfloor}
\def\wh{\widehat}
\newcommand{\defn}{\stackrel{def}{=}}
\newcommand{\txb}{\tau^{\eps,x}_{B^c}}
\newcommand{\tyb}{\tau^{\eps,y}_{B^c}}
\newcommand{\tilxb}{\tilde{\tau}^\eps_1}
\newcommand{\pxeps}{\mathbb{P}_x^{\eps}}
\newcommand{\non}{\nonumber}
\newcommand{\dist}{\mbox{dist}}

\newcommand{\Om}{\mathnormal{\Omega}}
\newcommand{\om}{\omega}
\newcommand{\vph}{\varphi}
\newcommand{\Del}{\mathnormal{\Delta}}
\newcommand{\Gam}{\mathnormal{\Gamma}}
\newcommand{\Sig}{\mathnormal{\Sigma}}

\newcommand{\tilyb}{\tilde{\tau}^\eps_2}
\newcommand{\beq}{\begin{eqnarray*}}
\newcommand{\eeq}{\end{eqnarray*}}
\newcommand{\beqn}{\begin{eqnarray}}
\newcommand{\eeqn}{\end{eqnarray}}
\newcommand{\ink}{\rule{.5\baselineskip}{.55\baselineskip}}

\newcommand{\bt}{\begin{theorem}}
\newcommand{\et}{\end{theorem}}
\newcommand{\deps}{\Del_{\eps}}
\newcommand{\dbl}{\mathbf{d}_{\tiny{\mbox{BL}}}}

\newcommand{\be}{\begin{equation}}
\newcommand{\ee}{\end{equation}}
\newcommand{\ac}{\mbox{AC}}
\newcommand{\hs}{\tiny{\mbox{HS}}}
\newcommand{\BB}{\mathbb{B}}
\newcommand{\VV}{\mathbb{V}}
\newcommand{\DD}{\mathbb{D}}
\newcommand{\KK}{\mathbb{K}}
\newcommand{\HH}{\mathbb{H}}
\newcommand{\TT}{\mathbb{T}}
\newcommand{\CC}{\mathbb{C}}
\newcommand{\ZZ}{\mathbb{Z}}
\newcommand{\SSS}{\mathbb{S}}
\newcommand{\EE}{\mathbb{E}}
\newcommand{\NN}{\mathbb{N}}
\newcommand{\MM}{\mathbb{M}}

\newcommand{\clg}{\mathcal{G}}
\newcommand{\clb}{\mathcal{B}}
\newcommand{\cls}{\mathcal{S}}
\newcommand{\clc}{\mathcal{C}}
\newcommand{\clj}{\mathcal{J}}
\newcommand{\clm}{\mathcal{M}}
\newcommand{\clx}{\mathcal{X}}
\newcommand{\cld}{\mathcal{D}}
\newcommand{\cle}{\mathcal{E}}
\newcommand{\clv}{\mathcal{V}}
\newcommand{\clu}{\mathcal{U}}
\newcommand{\clr}{\mathcal{R}}
\newcommand{\clt}{\mathcal{T}}
\newcommand{\cll}{\mathcal{L}}
\newcommand{\clz}{\mathcal{Z}}
\newcommand{\clq}{\mathcal{Q}}
\newcommand{\clo}{\mathcal{O}}

\newcommand{\cli}{\mathcal{I}}
\newcommand{\clp}{\mathcal{P}}
\newcommand{\cla}{\mathcal{A}}
\newcommand{\clf}{\mathcal{F}}
\newcommand{\clh}{\mathcal{H}}
\newcommand{\N}{\mathbb{N}}
\newcommand{\Q}{\mathbb{Q}}
\newcommand{\bfx}{{\boldsymbol{x}}}
\newcommand{\bfa}{{\boldsymbol{a}}}
\newcommand{\bfh}{{\boldsymbol{h}}}
\newcommand{\bfs}{{\boldsymbol{s}}}
\newcommand{\bfm}{{\boldsymbol{m}}}
\newcommand{\bff}{{\boldsymbol{f}}}
\newcommand{\bfb}{{\boldsymbol{b}}}
\newcommand{\bfw}{{\boldsymbol{w}}}
\newcommand{\bfz}{{\boldsymbol{z}}}
\newcommand{\bfu}{{\boldsymbol{u}}}
\newcommand{\bfell}{{\boldsymbol{\ell}}}
\newcommand{\bfn}{{\boldsymbol{n}}}
\newcommand{\bfd}{{\boldsymbol{d}}}
\newcommand{\bfbeta}{{\boldsymbol{\beta}}}
\newcommand{\bfzeta}{{\boldsymbol{\zeta}}}
\newcommand{\bfnu}{{\boldsymbol{\nu}}}
\newcommand{\bfvarphi}{{\boldsymbol{\varphi}}}

\newcommand{\curvz}{{\bf \mathpzc{z}}}
\newcommand{\curvx}{{\bf \mathpzc{x}}}
\newcommand{\curvi}{{\bf \mathpzc{i}}}
\newcommand{\curvs}{{\bf \mathpzc{s}}}
\newcommand{\blip}{\mathbb{B}_1}
\newcommand{\loc}{\text{loc}}

\newcommand{\BM}{\mbox{BM}}

\newcommand{\tac}{\mbox{\scriptsize{AC}}}
%\newcommand{\beginsec}{
% %\setcounter{equation}{0}
% \setcounter{lemma}{0} \setcounter{theorem}{0}
% \setcounter{corollary}{0} \setcounter{definition}{0}
% \setcounter{example}{0} \setcounter{proposition}{0}
% \setcounter{condition}{0} \setcounter{assumption}{0}
% %\setcounter{conjecture}{0} \setcounter{problem}{0}
% \setcounter{remark}{0} }

%\numberwithin{equation}{section} \numberwithin{lemma}{section}

%\renewcommand{\bibfont}{\footnotesize}

\begin{frontmatter}
\title{Diffusion Approximations for Load Balancing Mechanisms in Cloud Storage Systems	% \thanks{Research  supported in part by the National Science Foundation (DMS-1004418, DMS-1016441, DMS-1305120) and the Army Research
% 	 Office (W911NF-10-1-0158, W911NF- 14-1-0331)}
}

 \runtitle{Diffusion Approximations for Cloud Storage Systems}

\begin{aug}
\author{Amarjit Budhiraja  and Eric Friedlander\\ \ \\
}
\end{aug}

\today

\skp

\begin{abstract}
In large storage systems, files are often coded across several servers to improve reliability and retrieval speed. We study load balancing under the Batch Sampling routing scheme for a network of $n$ servers storing a set of files using the Maximum Distance Separable (MDS) code (cf. \cite{li2016mean}). Specifically, each file is stored in equally sized pieces across $L$ servers such that any $k$ pieces can reconstruct the original file. When a request for a file is received, the dispatcher routes the job into the $k$-shortest queues among the $L$ for which the corresponding server contains a piece of the file being requested. We establish a law of large numbers and a central limit theorem as the system becomes large (i.e. $n\to\iy$). For the central limit theorem, the limit process take values in $\bfell_2$, the space of square summable sequences. Due to the large size of such systems, a direct analysis of the $n$-server system is frequently intractable. The law of large numbers and diffusion approximations established in this work provide practical tools with which to perform such analysis. The Power-of-$d$ routing scheme, also known as the supermarket model, is a special case of the model considered here.

\noi {\bf AMS 2000 subject classifications:} Primary 60K35, 60H30, 93E20; secondary 60J28, 60J70, 60K25, 91B70.

\noi {\bf Keywords:} Mean field approximations, diffusion approximations, stochastic networks, cylindrical Brownian motion, propagation of chaos, cloud storage systems, supermarket model, MDS coding, Power-of-$d$.
\end{abstract}

\end{frontmatter}

\section{Introduction}
In the world of cloud-based computing, large data centers are often used for file storage. In these data centers, large networks of servers are used to store even larger sets of files. In order to improve reliability and retrieval speed, these files are often ``coded''. By coded, we mean that the file is broken down into smaller pieces which are stored on multiple servers. Consider the situation in which there are four servers and one file. One can store the entire file on one server but in such a configuration the file would be inaccessible if that server were to fail. In order to improve reliability, one can replicate the file across all four servers but such a method would require much more memory.  Suppose we instead split the file into halves, $A$ and $B$, and then store $A,\ B,\ A+B,\ A-B$ in each of the four servers, respectively. Then the original file can be constructed from any two pieces. One can extend this idea to the case where equally sized pieces of a file are stored across $L$ servers and any $k$ pieces can reconstruct the original file. This can be accomplished using the Maximum Distance Separable (MDS) code with parameters $(L,k)$ \cite{lin2004error}. The MDS code greatly improves reliability since $L-k+1$ servers must fail before the file becomes irretrievable, while only requiring enough total memory to store $L/k$ files. Given a coding scheme, one can consider load balancing mechanisms to improve file retrieval speed. In \cite{li2016mean}, two routing schemes, called Batch Sampling (BS) and Redundant Request with Killing (RRK), are considered. In BS routing, incoming jobs are routed to the $k$ shortest queues containing the file being requested, while in RRK routing jobs are routed to all servers containing the requested file and then removed from the queue (killed) once $k$ pieces of the file have been returned. The paper \cite{li2016mean} formally calculates the steady state ($T\to\iy$) queue length distribution in the large system limit $(n\to\iy)$ and gives simulation results for different values of $L$ and $k$ in both routing schemes.

In this work we are interested in developing a rigorous limit theory for such load balancing schemes for systems with MDS coding as $n$ becomes large. Specifically, we establish law of large numbers and diffusion approximations for such systems under an appropriate scaling, as $n\to\iy$. Such limit theorems provide useful model simplifications that can then be employed for approximate simulation of the large and complex $n$-server systems (see Section \ref{sec:example} for some numerical results). These limit theorems are also the first steps towards making rigorous the program initiated in \cite{li2016mean} of developing steady state approximations for such systems, with provable convergence properties as $n$ becomes large. 

We will focus in this work only on BS routing and leave analysis of the RRK scheme for future work.
 Specifically we consider a system with $n$ servers on which $I(n)$ files are stored using MDS coding with parameters $(L,k)$. A key assumption to our analysis is that the files are stored such that each combination of $L$ servers has exactly $c$ files.  We further assume that jobs arrive in the system at rate $n\lambda$ and request a file uniformly at random. This is another simplifying assumption on our model that roughly says that all files are in equal demand. These structural assumptions imply a convenient
 exchangeability property of the system which allows for the use of certain mean-field approximation techniques.
A single file request spawns $k$ jobs which are then routed into the $k$ shortest queues within the set of $L$ servers containing the file being requested. 
Each server  processes the jobs in their queue at rate $k$ according to the first-in-first-out (FIFO) discipline and processing times are mutually independent.
Regarding each server as a ``particle'', the above formulation describes an interacting particle system with \emph{simultaneous jumps}.
Note that the symmetry structure introduced above implies that every time
a file request arrives, it leads to a selection of $L$ servers uniformly at random (from which the $k$ servers with shortest queues are chosen). In particular this says that the well studied 
 ``Power-of-$d$'' routing scheme (also known as the ``supermarket model'') is a special case of the scheme considered here on taking $L=d$ and $k=1$. Direct analysis of such large and complex $n$-server systems is challenging even by simulation methods as
frequently the servers in networks of interest number in the hundreds of thousands with arrival rates of file requests
of similar order. The goal of this work is to develop  suitable approximate approaches to such systems. 

The starting point of our analysis is to consider, as the state descriptor, the empirical measure of the 
$n$ queue lengths rather than the  individual values of the queue lengths. Thus the state space
for our system will be the space $\mathcal{P}_n(\mathbb{N}_0)$ of probability measures on $\NN_0$
that assign weights in $\frac{1}{n}\NN_0$ to sets in $\NN_0$
rather than the space $\mathbb{R}_+^n$. With this formulation the state processes for all $n$-server systems
can be regarded as taking values in a common space $\cls\doteq\mathcal{P}(\mathbb{N}_0)$ (the space of probability measures
on $\mathbb{N}_0$). It follows from our symmetry assumptions that the state-evolution of the $n$-server system describes a pure-jump Markov process with values in $\mathcal{P}(\mathbb{N}_0)$ and thus one can bring to bear the theory of weak convergence of Markov processes to study the scaling limits as $n$ becomes large. In particular, in Theorem \ref{thm:LLN}
we prove a law of large numbers for the empirical measure process $(\pi^n(t))_{0\le t \le T}$ as $n \to \infty$.
We show that
$\pi^n$ converges to a deterministic function $\pi$ in $\DD([0,T]:\cls)$, where $\DD([0,T]:\cls)$ is the space of functions from $[0,T]$ to $\cls$ that are right continuous and have left limits, equipped with the usual Skorohod topology. 
Next we consider the fluctuation process $X^n\doteq\sqrt{n}(\pi^n-\pi)$.
This process can be regarded as taking values in the space of signed measures on $\NN_0$, however for an asymptotics analysis it is convenient to view it as taking values in the Hilbert space of square summable real sequences, $\bfell_2$.
The study of the asymptotics of these fluctuations is the subject of Theorem \ref{thm:mainResult} which shows that $X^n \doteq \sqrt{n}(\pi^n -\pi)$ converges in
$\mathbb{D}([0,T]:\bfell_2)$ to a $\bfell_2$-valued diffusion process.

% In particular, if we model the system through the empirical measure on the queue lengths, the resulting process is a pure-jump Markov process. Furthermore, under a fluid scaling, this state process converges to a deterministic system of ordinary differential equations (ODEs) which approximate the nominal dynamics of the system (see Theorem \ref{thm:LLN}). In addition, the appropriately centered and scaled state process can be shown to converge to a system of stochastic differential equations (SDEs) (see Theorem \ref{thm:mainResult}). This result provides tools with which a system engineer can efficiently calculate a variety of metrics pertaining to the performance of the system through simulation of the limit diffusion.

Load balancing mechanisms similar to the type considered here have been studied in many works. Specifically, the Join-the-Shortest-Queue (JSQ), Join-the-Idle-Queue (JIQ), and Power-of-$d$ routing schemes have garnered quite a bit of attention (see \cite{vvedenskaya1996queueing, eschenfeldt2015join,mukherjee2015universality,mitzenmacher2001power,bramson2012asymptotic,stolyar2015pull, graham2000chaoticity} and  references therein). The papers \cite{vvedenskaya1996queueing, mitzenmacher2001power} first analyzed the Power-of-$d$ routing scheme, showing that in steady-state the fraction of queues with lengths exceeding $m$
 decay super-exponentially in $m$, a large improvement over the exponential rate for the setting where jobs are routed to servers uniformly at random. The paper \cite{graham2000chaoticity} establishes a functional law of large numbers for $\pi^n$
on $\DD([0,T]:\cls)$ in the Power-of-$d$ routing scheme using characterization results for nonlinear martingale problems.
In \cite{eschenfeldt2015join}, the authors derive a diffusion approximation for the JSQ routing policy in the large-system limit under heavy-traffic scaling. It is shown that the  limit can be characterized through a two-dimensional diffusion. In \cite{mukherjee2015universality}, it is shown that the JIQ routing scheme yields the same 
diffusion approximation as the JSQ routing scheme. In both of these works, the  diffusion approximations are derived under the same scaling regime as considered here. However, unlike for the JSQ and JIQ routing regimes where the diffusion approximations can be described through two-dimensional processes, in the current work
 the limit is an infinite dimensional diffusion described as an 
 $\bfell_2$-valued process driven by a cylindrical Brownian motion. 
As was noted earlier, 
the Power-of-$d$ scheme is a special case of our results and thus Theorems \ref{thm:LLN} and
\ref{thm:mainResult} also provide law of large numbers and diffusion approximations for this classical load balancing scheme (see Corollaries \ref{cor:pdLLN} and \ref{cor:pddiff}).
In particular Corollary \ref{cor:pdLLN} recovers the law of large numbers established in \cite{graham2000chaoticity}.
To the best of our knowledge,  limit theorems giving diffusion approximations for Power-of-$d$ schemes have not been studied previously.

Diffusion approximation methods have been used extensively in stochastic network theory. In particular they have been very useful in the study of critically loaded stochastic processing networks (see \cite{kushner2013heavy,harrison1988brownian,atar2014asymptotic, bellwill01, DaiLin, wwbook, AA2, BGL} and references therein). For such systems the key mathematical tool is the functional central limit theorem for renewal processes which provides Brownian motion approximations for  a finite collection of centered renewal processes with rates approaching infinity. The scaling regime and mathematical tools that are relevant for the analysis in the current context are quite different from those used in the above works. In particular, here the number of nodes approach $\infty$ and the tools for proving convergence come from martingale problems and Markov process theory.  
A similar scaling regime was considered in \cite{budhiraja2016diffusion} for certain systems motivated by ad hoc wireless network models introduced in \cite{antunes2008stochastic}. A key simplifying feature there is that the 
state space of an individual queue is a finite set.  Consequently the limit diffusion in \cite{budhiraja2016diffusion} is finite dimensional and thus, for diffusion approximations, classical convergence 
theorems from \cite{kurtz1971limit,joffe1986weak} can be invoked. 
In contrast, the queue length  processes in this work are unbounded, taking values in $\mathbb{N}_0$, and thus one needs to study diffusion approximations in an infinite dimensional state space, namely the Hilbert space $\bfell_2$.
The proofs employ  appropriate criteria for tightness and characterization results for Hilbert space-valued stochastic processes.  

A basic assumption in our analysis of the fluctuations around the law of large number limit  (see statement of Theorem
\ref{thm:mainResult}) is a uniform (in $n$) bound on the second  moment of the empirical measure at time $0$.
This condition is not very stringent as in practice one may consider systems starting from empty 
or with finitely many jobs (independent of $n$). 
We argue that these integrability properties at time $0$
 propagate through to any finite future time $T$. 
Tightness of the scaled fluctuation processes $X^n$ which is shown by establishing, uniform in $n$, second moment bounds (on $X^n$) and by employing 
criteria for tightness of Hilbert space-valued semimartingales (cf. \cite{joffe1986weak}, \cite{metivier1982semimartingales}), relies on these integrability properties.
Another ingredient in the proof of tightness is a suitable Lipschitz property of the map $F$ introduced in \eqref{eqn:codingF}  that enables the use of a Gronwall argument. For this argument one needs a Lipschitz estimate in the $\bfell_2$ norm, however, it is not  clear that $F$, as a map from $\bfell_2$ to $\bfell_2$, is Lipschitz. We instead restrict attention to a smaller space
$$\clv_M \doteq \left\{r \in \bfell_2: r_i\ge 0,  \sum_{i=0}^\iy r_i=1,\;  \sum_{i=0}^\iy i r_i \le M\right\}$$
and argue that for each $M$, the map $F$ is Lipschitz from $\clv_M$ to $\bfell_2$. This `local' Lipschitz property plays an important role in the proof of Proposition \ref{prop:diffTightness}.

For characterization of limit points in the proof of the central limit theorem, one needs to argue that the
associated stochastic differential equation (SDE) in $\bfell_2$ (see \eqref{eqn:limitSDE}) has a unique weak solution in an appropriate class of processes.
It turns out that arguing this uniqueness among adapted processes with paths in $\mathbb{C}([0,T]:\bfell_2)$ (the space of continuous functions from
$[0,T]$ to $\bfell_2$) is not straightforward due to a lack of suitable regularity of the function
$G$ introduced in \eqref{eqn:opDFdef}. In particular, once more, the Lipschitz property of the map $x \mapsto G(x,\pi)$ (for a fixed $\pi \in \mathcal{P}(\mathbb{N}_0)$)
from $\bfell_2$ to itself is not immediate. The key observation here is that this map is Lipschitz when restricted
to the space
\begin{equation*}
\ti\bfell_2\doteq\{x\in\bfell_2:\sum_{j=0}^\iy j^2x_j^2<\iy,\sum_{j=0}^\iy x_j = 0\}.
\end{equation*}
This observation, together with the property that the limit points $X$ of $X^n = \sqrt{n}(\pi^n-\pi)$.
satisfy $X(t) \in \ti\bfell_2$ for all $t\ge 0$ almost surely, is key to the characterization of the limit points as
the unique solution of  the SDE \eqref{eqn:limitSDE} in a suitable class (see Proposition \ref{prop:DiffUniqueness}).

The paper is organized as follows. In Section \ref{sec:model} we give a precise mathematical formulation of our model and a statement of our main results. Specifically, Theorem \ref{thm:LLN} provides the convergence in probability of the empirical measure process  in $\mathbb{D}([0,T]: \cls)$  to the unique solution of the ODE defined in \eqref{eqn:ODE}. In Theorem \ref{thm:mainResult}, we present the main diffusion approximation result. This result says that the sequence of centered and scaled processes $X^n$, defined in \eqref{eqn:Xndef}, converges to the unique solution (in a suitable class) of the $\bfell_2$-valued SDE, driven by a cylindrical Brownian motion, given in \eqref{eqn:limitSDE}. In Section \ref{sec:supmarkmod} we record the corollaries  of these results to the special setting of power-of-$d$ schemes. We then proceed to the proofs of Theorems \ref{thm:LLN} and \ref{thm:mainResult}.
In Section \ref{sec:semiMartRep} we give a 
convenient  representation of the state processes through a countable number of time-changed unit rate Poisson processes.
Such Poisson representations have been used extensively (cf. \cite{kurtz1980representations,kang2014central,anderson2015stochastic}) in the study of diffusion approximations for pure jump processes.
Using this we obtain a  semimartingale decomposition (see \eqref{eqn:semimartrep})
that is central to our analysis. Section \ref{sec:LLN} is devoted to the proof of Theorem \ref{thm:LLN}.  In Section \ref{sec:piTight} we prove tightness of the sequence of state processes  
$\{\pi^n\}_{n\in\NN}$ (see Proposition \ref{prop:tightness}) and the proof of  Theorem \ref{thm:LLN} is completed  in Section \ref{sec:piLimitPoints}. 
Section \ref{sec:diffApprox} proves Theorem \ref{thm:mainResult}. 
In Section \ref{sec:mombds} we prove the propagation of integrability properties that was discussed earlier and in Section \ref{sec:diffTightness} (see Proposition \ref{prop:diffTightness}) we prove the key tightness property for the sequence of processes $\{X^n\}_{n\in \NN}$ which relies on the Lipschitz property of $F$, in the $\bfell_2$ norm, on $\clv_M$ (Lemma \ref{lem:lipschitz2}). Theorem \ref{thm:mainResult} is then proved in Section \ref{sec:diffConv}. Finally, in Section \ref{sec:example}, we present some  numerical results. In particular, we use our results to give numerical confidence intervals for several performance measures of interest and compare the results to those obtained from a direct simulation of the corresponding $n$-server systems.

\subsection{Notation}
The following notation will be used. 
Fix $T < \infty$. All stochastic processes will be considered over the time horizon $[0,T]$.  
We will use the notations $(X_t)_{0\leq t\leq T}$ and $(X(t))_{0\leq t\leq T}$ interchangeably for stochastic processes.
The space of probability measures on a Polish space $\SSS$, equipped with the topology of weak convergence, will be denoted by $\clp(\SSS)$.  When $\SSS = \NN_0$ we will metrize $\clp(\SSS)$ with the metric $d_0$ defined as
$$d_0(\mu, \nu) \doteq \sum_{j=0}^{\infty} \frac{|\mu(j)-\nu(j)|}{2^j}, \; \mu, \nu \in \clp(\NN_0).$$
For $\SSS$-valued random variables $X$, $X_n$, $n\ge 1$, convergence in distribution of $X_n$ to $X$ as $n\to \infty$ will be denoted as $X_n \Rightarrow X$.
%A convenient metric for this topology 
%is the bounded-Lipschitz metric $\dbl$ defined as
%$$
%\dbl(\nu_1, \nu_2) = \sup_{f \in \blip} | \langle f , \nu_1 - \nu_2\rangle|, \nu_1, \nu_2 \in \clp(\SSS),$$
%where  $\blip$ is the collection of  all Lipschitz functions $f$ that are bounded by $1$ and such that the corresponding Lipschitz constant is bounded by $1$ as well; and 
%$\langle f, \mu\rangle = \int f d\mu$ for a signed measure $\mu$ on $\SSS$ and $\mu$-integrable $f: \SSS \to \RR$.
%For a function $f: [0,T] \to \RR^k$, $\|f\|_{*,t} \doteq \sup_{0\le s \le t}\|f(s)\|$, $t \in [0,T]$.  Also, for $\mu_i:[0,T] \to \clp(\SSS)$, $i=1,2$,
%$$\dbl(\mu_1, \mu_2)_{*,t} = \sup_{0\le s \le t}\dbl(\mu_1(s), \mu_2(s)).$$
%The Borel $\sigma$-field on a Polish space $\SSS$ will be denoted as $\clb(\SSS)$. 
The space of functions that are right continuous with left limits (RCLL) from $[0,T]$  to $\SSS$
will be denoted as $\DD([0,T]:\SSS)$  and equipped with the usual Skorohod topology.
Similarly $\CC([0,T]:\SSS)$  will be the space of continuous functions
from $[0,T]$  to $\SSS$,  equipped with the  uniform topology.
We will usually denote by $\kappa, \kappa_1, \kappa_2, \cdots$, the constants that appear in various estimates within a proof. The values of these constants may change from one proof to another and, unless stated otherwise, will take values in the set $(0,\iy)$.
Let $\bfell_2 = \{(a_j)_{j=0}^\iy|\sum_{j=0}^\iy a_j^2<\iy\}$  be the space of square summable real sequences. This space is a Hilbert space with inner product
\begin{equation*}
\lan x,y\ran_2 = \sum_{j=0}^\iy x_jy_j.
\end{equation*}
We denote the corresponding norm as  $\|\cdot\|_2$.
Similarly, $\bfell_1=\{(a_j)_{j=0}^\iy|\sum_{j=0}^\iy|a_j|<\iy\}$ and $\|\cdot\|_1$ is the norm on this Banach space.
The Hilbert-Schmidt norm of a Hilbert-Schmidt operator $A$  on $\bfell_2$ will be denoted $\|A\|_{\hs}$ (cf. Appendix \ref{sec:HSInfo}).
%The transpose of an operator $M$ will be denoted as $M^T$, the trace will be denoted as $\Tr(M)$ . 
We denote by $I$ the identity operator.
For a Hilbert Space $\HH$,  $\clm^2_T(\HH)$ will denote the space of all $\HH$-valued continuous, square integrable martingales $M$, such that $M(0)=0$. For a real number $a$, $(a)_+$ will denote the positive part of $a$.
% This space will be equipped with the usual vague topology, namely, the weakest topology such that for every $f\in\CC_b(\SSS)$ with compact support,
%\begin{align*}
%\nu\mapsto\int_\SSS f(u)\nu(du),\ \nu\in\clm(\SSS),
%\end{align*}
%is continuous.

\section{Model Description and Main Result}\label{sec:model}
We consider a system with $n$ servers each with its own infinite capacity queue. In all, there are $I(n)$ equally sized files  stored over the $n$ servers. Each file is stored in equally sized pieces at $L$ servers such that any $k$ pieces can reconstruct the original file. The files are distributed such that each combination of $L$ servers has exactly $c$ files. This, in particular, implies $I(n)=c\binom{n}{L}$. Jobs arrive from outside  according to a Poisson process with rate $n\lambda$ and request one of the $I(n)$ files uniformly at random. Such a request corresponds to selection of one of the $\binom{n}{L}$ sets of $L$ servers, uniformly at random, which is the set of servers containing the pieces of the requested file. The job is then routed to the $k$ shortest queues among this set of $L$ servers. Each server processes queued jobs according to the first-in-first-out (FIFO) discipline. Processing times at each server are mutually independent and exponentially distribution with mean $k^{-1}$. 

Let $Q^n(t)=\{Q^n_i(t)\}_{i=1}^n$ where $Q^n_i(t)$ represents the length of the $i$-th queue at time $t$ and let $\pi^n(t)=\{\pi^n_i(t)\}_{i\in\NN_0}$ where $\pi_i^n(t)$ represents the proportion of queues with length exactly $i$ at time $t$. This can explicitly be written as
\begin{equation}
\pi^n_i(t)=\frac{1}{n}\sum_{j=1}^n1_{\{Q^n_j(t)=i\}}.\label{eq:eq102}
\end{equation}
We will assume for simplicity that $Q^n(0)=q^n$ is nonrandom and thus $\pi^n(0) = \frac{1}{n}\sum_{j=1}^n1_{\{q^n_j=i\}}$ is nonrandom as well.
We identify $\clp(\NN_0)$ with the infinite dimensional simplex $\cls=\{s\in\RR_+^\iy|\sum_{i=0}^\iy s_i=1\}$ and let $\cls_n=\frac{1}{n}\NN_0^\iy\cap\cls$. It follows that $\pi^n(t)\in\cls_n$ for all $t\in[0,T]$. Let $\Sigma=\{\ell=(\ell_i)_{i=1}^L\in\NN^L_0|\ell_1\leq\ell_2\leq\cdots\leq\ell_L\}$ and for $\ell\in\Sigma$ define $\rho_i(\ell)\doteq\sum_{j=1}^L1_{\{\ell_j=i\}},\ i\in\NN_0$. Roughly speaking, $\Sigma$ will represent the set of possible states for $L$ selected queues arranged by non-decreasing queue length. Note that each file will be stored at $L$ servers and that at any given time $t$ the queue lengths of these $L$ servers (up to a reordering) will correspond to an element in $\Sigma$. We will  refer to the elements of $\Sigma$ as ``queue length configurations''. 
Given a configuration $\bfell\in\Sigma$, $\rho_i(\ell)$ gives the number of queues of length $i$ (among the $L$ selected).
From the above description of the system it follows that
the empirical measure process, $\pi^n(t)$, is a continuous time Markov chain with state space $\cls_n$ and generator
\begin{equation}\label{eqn:generator}
\begin{aligned}
\cll^n f(r)
&= \frac{n\lambda}{\binom{n}{L}}\sum_{\ell\in\Sigma}\left(\prod_{i=0}^\iy\binom{nr_i}{\rho_i(\ell)}\right)\left[f\left(r+\frac{1}{n}\Del_\ell\right)-f(r)\right]\\
&\qquad + k\sum_{i=1}^\iy nr_i\left[f\left(r+\frac{1}{n}(e_{i-1}-e_i)\right)-f(r)\right],
\end{aligned}
\end{equation}
for $f:\cls_n\to\RR$ where
\begin{equation}\label{eqn:deldef}
\Del_\ell \doteq \sum_{i=1}^k e_{\ell_i+1}-\sum_{i=1}^k e_{\ell_i}
\end{equation}
and for $y\in\NN_0,\ e_y\in\bfell_2$ is a vector with 1 at the $y$-th coordinate and $0$ elsewhere. Here we use the standard conventions that $0^0=\binom{0}{0}=0!=1$, and $\binom{a}{b}=0$ when $a<b$. 
The above generator can be understood as follows. A typical term in the second expression corresponds to a jump as a result of a server, with exactly $i$ jobs queued, completing a job.
The term in the square brackets gives the change in value of $f$ as a result of such a jump and the prefactor $knr_i$ corresponds to the fact that servers process jobs at rate $k$ and there are in all $nr_i$ queues (prior to the jump) with exactly $i$ jobs.
The first expression in \eqref{eqn:generator} corresponds to a jump resulting from an arrival of a job to the system.
Typically, such an arrival makes a request for $L$ servers with queue length configuration
$\ell_1\le \ell_2 \le \cdots\le \ell_L$ and results in the jump $\frac{1}{n}\Del_\ell$. 
The sum in \eqref{eqn:deldef} only goes up to $k$ (instead of $L$) since only the smallest $k$ queues are affected by such a jump.
Since prior to the jump, there are $n r_i$ queues with exactly $i$ jobs, the overall rate associated with the configuration $\ell = \{\ell_1\le \ell_2 \le \cdots\le \ell_L\} \in \Sigma$ equals
$$\frac{n\lambda}{\binom{n}{L}}\left(\prod_{i=0}^\iy\binom{nr_i}{\rho_i(\ell)}\right).$$

In our setting the first entry in an element of $\bfell_2$ will typically correspond to the number of empty queues and thus we refer to it as the ``0-th'' coordinate and any $r\in\bfell_2$ will
correspond to a vector of the form $(r_0, r_1, \ldots)$. For notational convenience, for $r\in\bfell_2$ we set $r_{-1}\doteq 0$.

%Since we are modeling a queueing system the first entry in an element of $\bfell_2$ will typically relate to empty queues. We refer to this coordinate as the ``0-th'' coordinate and, for $r\in\bfell_2$, write $r_0$. 
 
The main results in this work provide scaling limits for $\pi^n$. We first present the law of large numbers which describes the nominal state of the system for large $n$. 
Define, for $r\in\bfell_1$,
\begin{equation}\label{eqn:codingF}
F(r)
\doteq  \lambda L!\sum_{j=0}^\iy \bar{\zeta}^\del(j,r)e_j + k\sum_{j=0}^\iy[r_{j+1}-r_j]e_j+r_0e_0
\end{equation}
where
\begin{equation*}
 \bar{\zeta}^\del(j,r) \doteq \bar{\zeta}(j-1,r)-\bar{\zeta}(j,r)
\end{equation*}
and, adopting the convention that $\sum_{i=b}^ax_i=0$ for $a<b$,
\begin{equation}\label{eqn:zetabardef}
\bar{\zeta}(j,r)
\doteq \sum_{i_1=0}^{k-1}\frac{\left(\sum_{m=0}^{j-1}r_m\right)^{i_1}}{i_1!}\sum_{i_2=1}^{L-i_1}[i_2\wedge(k-i_1)]\frac{(r_j)^{i_2}}{i_2!}\frac{\left(\sum_{m=j+1}^{\iy}r_{m}\right)^{L-i_1-i_2}}{(L-i_1-i_2)!}.
\end{equation}
For $j\geq 0$, the quantities $k[r_{j+1}-r_j]$ in \eqref{eqn:codingF} roughly represent the rate at which the $j$-th coordinate of the state changes (in the limit) as a result of job-completions while the quantity
$\lambda L! (\bar{\zeta}(j-1,r)-\bar{\zeta}(j,r))$ represents a similar quantity as a result of job-arrivals. The various terms in \eqref{eqn:zetabardef} can be interpreted as follows. An arrival
to a queue with $j$ jobs implies that a queue length configuration vector $\ell = \{\ell_1\le \ell_2 \le \cdots\le \ell_L\}$ was selected which has the property that at least one of the $k$ smallest
$\ell_i$'s equals $j$, or equivalently, exactly $i_1$ ($i_1 =0, 1, \ldots , k-1$) of the smallest $L$ selected are less than $j$, $i_2$ ($i_2=1, \ldots L-i_1$) of these are equal to $j$, and $L-i_1-i_2$
are greater than $j$. The three ratios in \eqref{eqn:zetabardef} are contributions from these three types of queues. The term $[i_2\wedge(k-i_1)]$ is from the fact that only the smallest $k$ of the $L$ queues are affected. 

Also observe that for some $c_{\zeta} \in (0,\infty)$
\begin{equation}\label{eqn:czetabound}
	\bar \zeta(j,r) \le c_{\zeta} r_j \mbox{ for all } j \in \NN_0 \mbox{ and } r= (r_j)_{j=0}^{\infty}	\in \cls. 
\end{equation}
Thus the infinite sum in \eqref{eqn:codingF} is well defined since $\sum_{j=0}^{\infty} r_j =1$ and consequently $F$ is a well defined map from $\cls$ to $\bfell_1$.
A similar estimate shows that $F$ is a well defined map from $\bfell_1$ to $\bfell_1$ and $\sum_{j=0}^\iy F_j(r)=0$ for all $r\in\bfell_1$.

Consider the system of ODEs
\begin{equation}\label{eqn:ODE}
\dot{\pi}(t) = F(\pi(t)),\qquad \pi(0) = \pi_0
\end{equation}
where $F$ is defined in \eqref{eqn:codingF} and $\pi_0\in\cls$. The solution of the equation is a continuous map  $\pi: [0,T] \to \cls$ such that
\begin{equation}
	\label{eq:eq324}
\pi(t) = \pi_0 + \int_0^t F(\pi(s)) ds, \; t \in [0,T]
\end{equation}
where the integral on the right side is the classical Bochner integral which is well defined  since, from \eqref{eqn:codingF} and \eqref{eqn:czetabound},
\begin{equation}
	\label{eq:eq305}
\sup_{0\le s \le T}\|F(\pi(s))\|_1 \le \sup_{r \in \cls} \|F(r)\|_{1} < \infty .
\end{equation}
Equation \eqref{eqn:ODE} will characterize the law of large number limit of $\pi^n$.
% , however it is not immediately obvious that $F$, as a map from
% $\cls \to \bfell_2$, is Lipschitz which makes it hard to argue the well-posedness of \eqref{eqn:ODE} on $\CC([0,T]:\cls)$.
% We will instead establish unique solvability on a smaller space, introduced below, which suffices for characterization of the limits of $\pi^n$.  

The following result on the wellposedness of \eqref{eqn:ODE} will be shown in Section \ref{sec:piLimitPoints}.

\begin{proposition}\label{prop:uniqueness}
Let $\pi_0 \in \cls$. Then there exists a $\pi\in \CC([0,T]:\cls)$ that solves \eqref{eqn:ODE}. Furthermore, if $\pi, \tilde \pi$ are two elements of
$\CC([0,T]:\cls)$ solving \eqref{eqn:ODE} with $\pi(0)=\ti\pi(0)=\pi_0$, then $\pi=\tilde \pi$.
\end{proposition}

The next theorem gives a law of large numbers for the sequence $\{\pi^n\}_{n\in\NN}$.
Recall we take $\pi^n(0)$ to be nonrandom.
\begin{theorem}\label{thm:LLN}
Suppose that $\pi^n(0) \to \pi_0$, in $\cls$, as $n \to \infty$. 
 Then $\pi^n\to\pi$, in probability, in $\DD([0,T]:\cls)$ where $\pi$ is the unique solution of \eqref{eqn:ODE} in $\CC([0,T]:\cls)$.
\end{theorem}
Proof of  Theorem \ref{thm:LLN} will be given in Section \ref{sec:piLimitPoints}. 

Our second main result studies the fluctuations of $\pi^n$ from its law of large number limit.
Consider
\begin{equation}\label{eqn:Xndef}
X^n(t) = \sqrt{n}[\pi^n(t)-\pi(t)], \; t \in [0,T].
\end{equation} 
where $\pi^n$ is the state process introduced in \eqref{eq:eq102}  and $\pi$ is the unique solution of \eqref{eqn:ODE} in $\CC([0,T]:\cls)$.

We will show that, under conditions, $X^n$ converges in distribution in $\DD([0,T]: \bfell_2)$ to a stochastic process that can be  characterized 
as the solution of a stochastic differential equation (SDE) of the following form.
\begin{equation}\label{eqn:limitSDE}
dX(t) = G(X(t),\pi(t))dt+a(t)dW(t), \qquad X(0)=x_0.
\end{equation}
The equation is again interpreted in the integrated form,
\begin{equation}\label{eqn:limitSDEb}
X(t) = x_0 + \int_0^t G(X(s),\pi(s))ds + \int_0^t a(s)dW(s), \qquad \; t \in [0,T].
\end{equation}
In the above equations, $a$ is a measurable map from $[0,T]$ to the space of Hilbert-Schmidt operators from $\ell_2$ to $\ell_2$ such that $\int_0^T \|a(t)\|_{\hs}^2 dt <\infty$, where $\|\cdot\|_{\hs}$
denotes the Hilbert-Schmidt norm (see Appendix \ref{sec:HSInfo}), and $W$ is a $\bfell_2$-cylinderical Brownian motion. Precise definitions are given in Appendix \ref{sec:CylWeiner}, but roughly speaking, $W$ can be identified with an iid 
sequence $\{\beta_i\}_{i\in \NN_0}$ of standard real Brownian motions over $[0,T]$ and the stochastic integral $\int_0^t a(s) dW(s)$ represents a $\bfell_2$-valued Gaussian martingale $M(t)$ given as
\begin{equation}\label{eq:eq410}
M_i(t) = \sum_{j=0}^{\infty} \int_0^t A_{ij}(s) d\beta_j(s), \; t \in [0,T],\; i \in \NN_0,
\end{equation}
where $A_{ij}(s) = \lan e_i, a(s)e_j\ran_2$, $s\in [0,T]$, $i,j \in \NN_0$.
We refer the reader to Chapter 4 of \cite{da2014stochastic} for construction and properties of the stochastic integral in \eqref{eqn:limitSDEb}.  The Hilbert-Schmidt and integrability property of $a$  ensure that
the infinite sum in \eqref{eq:eq410} converges.  The operator $a(t)$ is determined from the system parameters and the law of large number limit $\pi$
in Theorem \ref{thm:LLN} as the symmetric square root of the following non-negative trace class operator
\begin{equation}\label{eqn:squarematrix}
\Phi(t)
\doteq \lambda L!\sum_{\ell\in\Sigma}\Del_\ell\Del_\ell^T\prod_{i=0}^\iy\frac{\pi_i(t)^{\rho_i(\ell)}}{\rho_i(\ell)!}+k\sum_{i=1}^\iy (e_{i-1}-e_i)(e_{i-1}-e_i)^T\pi_i(t).
\end{equation}
The trace class property of $\Phi(t)$ and the integrability of the squared Hilbert-Schmidt norm of $a(t)$ are shown in Lemma \ref{lem:traceOp}. 
Define the space $\ti\bfell_2 \subset \bfell_2$  as
\begin{equation}\label{eqn:tibfelldef}
\ti\bfell_2\doteq\{x\in\bfell_2:\sum_{j=0}^\iy j^2x_j^2<\iy,\sum_{j=0}^\iy x_j = 0\}.
\end{equation}
In \eqref{eqn:limitSDE} $G$ is a map from $\ti\bfell_2 \times \cls$ to $\bfell_2$  defined as
\begin{equation}\label{eqn:opDFdef}
G_i(x,r)
\doteq \frac{\partial}{\partial u}F_i(r+ux)\Big|_{u=0}\qquad i\in\NN_0,\ u\in \RR.
\end{equation}
One of the difficulties in the analysis is that $G$ as a map from $\bfell_2 \times \cls$ to $\bfell_2$ is not well behaved and we need to restrict attention to the smaller space 
$\ti\bfell_2 \times \cls$ in order to get unique solvability of \eqref{eqn:limitSDE}. 
Note that under the condition $\sum_{j=0}^\iy j^2x_j^2<\iy$, the series $\sum_{j=0}^\iy |x_j|<\iy$ and thus the series $\sum_{j=0}^\iy x_j$ is convergent. Additionally, the right side of \eqref{eqn:opDFdef} is well defined for every $x\in\ti\bfell_2$ and $r\in\cls$, since for each $j\in\NN_0$ and $r\in\bfell_1$ with $\sum_{i=0}^\iy r_i=1$, $r\mapsto F_j(r)$ is a polynomial in $(r_0,r_1,\ldots,r_{j+1})$ given as
\begin{equation*}
F_j(r)= \lambda L![\bar{\zeta}(j-1,r)-\bar{\zeta}(j,r)]+k(r_{j+1}-r_j)
\end{equation*}
where
\begin{equation*}
\bar{\zeta}(j,r)
= \sum_{i_1=0}^{k-1}\frac{\left(\sum_{m=0}^{j-1}r_m\right)^{i_1}}{i_1!}\sum_{i_2=1}^{L-i_1}[i_2\wedge(k-i_1)]\frac{(r_j)^{i_2}}{i_2!}\frac{\left(1-\sum_{m=0}^{j}r_{m}\right)^{L-i_1-i_2}}{(L-i_1-i_2)!}.
\end{equation*}
Also, from \eqref{eqn:codingF} and \eqref{eqn:zetabardef} it is easily checked that there is a $c\in(0,\iy)$ such that for all $x\in\ti\bfell_2$ and $r\in\cls$
\begin{equation*}
|G_i(x,r)|\leq c\left[|x_{i-1}|+|x_i|+|x_{i+1}|+(r_{i-1}+r_i)\sum_{m=0}^\iy|x_{m}|\right].
\end{equation*}
This in particular implies that $G(x,r)\doteq (G_i(x,r))_{i\in\NN_0}\in\bfell_1\subset\bfell_2$ for all $(x,r)\in\ti\bfell_2\times\cls$.

The following result shows the well-posedness of \eqref{eqn:limitSDEb}. The definition of an $\bfell_2$-cylindrical Brownian motion is given in Section \ref{sec:CylWeiner}.
\begin{proposition}\label{prop:DiffUniqueness}
	There exists a filtered probability space $(\Om,\clf,\PP,\{\clf_t\})$ on which is given 
 a $\bfell_2$-cylindrical Brownian motion  $W$ and a continuous $\{\clf_t\}$-adapted process $(X(t))_{0\leq t\leq T}$ with sample paths in $\CC([0,T]:\bfell_2)$ that satisfies the integral equation \eqref{eqn:limitSDEb} and is such that $X(t)\in\ti\bfell_2\subset\bfell_2$ for all $t\in[0,T]$ almost surely. Furthermore if $\{\ti X_t\}_{0\leq t\leq T}$ is another such process then $\ti X_t=X_t$ for all $t\in[0,T]$, almost surely.
\end{proposition} 
The above result establishes weak existence and pathwise uniqueness of \eqref{eqn:limitSDEb}. By a standard argument (cf. \cite[Section IV.1]{IkedaWatanabe})  it follows that
\eqref{eqn:limitSDEb}  has a unique weak solution. 
We can now present our main result on fluctuations of $\pi^n$. Recall that $X^n(0)=\sqrt{n}(\pi^n(0)-\pi_0)$ is deterministic.
\begin{theorem}\label{thm:mainResult}
Suppose $\sup_{n\in\NN}\sum_{j=0}^\iy j^2\pi^n_j(0)<\iy$ and $\pi^n(0) \to \pi_0$ in $\cls$ as $n\to \infty$.
Let $\pi$ be the unique solution of \eqref{eqn:ODE} and, with $X^n$ defined as in \eqref{eqn:Xndef}, $X^n(0)\to x_0$ in $\bfell_2$. In addition, suppose that
\begin{equation}\label{eqn:XinitBound}
\sup_{n\in\NN}\sum_{j=0}^\iy j^2(X^n_j(0))^2<\iy.
\end{equation} 
Then $X^n\Rightarrow X$ in $\DD([0,T]:\bfell_2)$ where $X$ is the unique weak solution to \eqref{eqn:limitSDE} given by Proposition \ref{prop:DiffUniqueness}.
\end{theorem}
Proposition \ref{prop:DiffUniqueness} and Theorem \ref{thm:mainResult} will be proved in Section \ref{sec:diffApprox}. In Section \ref{sec:example} we will describe how Theorems \ref{thm:LLN} and \ref{thm:mainResult} can be used for numerical computation of various performance measures using simulation of diffusion processes.

\subsection{Supermarket Model}
\label{sec:supmarkmod}
Consider a system of $n$ servers, each with its own queue. Jobs arrive in the system according to a Poisson process with rate $n\lambda$. When a job enters the system, $d$ servers are chosen uniformly at random and the job is routed to the shortest of the $d$ selected queues. All servers process jobs according to the FIFO discipline. Service times are mutually independent and exponentially distributed with mean 1. This model has been well studied and is known as Power-of-$d$ routing or the ``Supermarket Model'' (see \cite{vvedenskaya1996queueing,mitzenmacher2001power,graham2000chaoticity}). The model is a special case of the system considered in the current work, corresponding to  $L=d$ and $k=1$. Theorems \ref{thm:LLN} and \ref{thm:mainResult} then provide, as corollaries, the following law of large numbers and central limit theorem for the Power-of-$d$ routing scheme. 

Define by $\pi^n_d$ the empirical measure process of queue lengths in the Power-of-$d$ system. For $r\in\bfell_1$, define 	
\begin{equation*}
F_d(r)
\doteq  \lambda \left[\sum_{i=1}^d\binom{d}{i}r_{j-1}^i\left(\sum_{m=j}^\iy r_m\right)^{d-i}-\sum_{i=1}^d\binom{d}{i}r_j^i\left(\sum_{m=j+1}^\iy r_m\right)^{d-i}\right]e_j + \sum_{j=0}^\iy[r_{j+1}-r_j]e_j.
\end{equation*}
The following is a direct corollary of Theorem \ref{thm:LLN}.
\begin{corollary}\label{cor:pdLLN}
Suppose that $\pi_d^n(0) \to \pi_d(0)$, in $\cls$, as $n \to \infty$.  Then
$\pi_d^n\to\pi_d$, in probability, in $\DD([0,T]:\cls)$ where $\pi_d$ is the unique solution in $\CC([0,T]:\cls)$ to the following ODE
\begin{equation*}
\dot{\pi}_d(t) = F_d(\pi_d(t)),\qquad \pi_d(0) = \pi_0.
\end{equation*}
\end{corollary}
\begin{remark}
This result  has been established in \cite{graham2000chaoticity} (see Theorem 3.4 therein). In particular, it is easy to verify that $v_m(t)\doteq\sum_{j=m}^\iy(\pi_d(t))_j$ is the same function as in (3.9) of \cite{graham2000chaoticity} (see also \cite{vvedenskaya1996queueing}).
\end{remark}
Our second corollary studies the fluctuations of $\pi^n_d$ from its law of large number limit. Consider
\begin{equation*}
X^n_d(t) = \sqrt{n}[\pi^n_d(t)-\pi_d(t)], \; t \in [0,T].
\end{equation*} 
Analogous to $a(t)$ introduced in \eqref{eqn:limitSDE}, let $a_d(t)$ be the symmetric square root of the following non-negative operator
\begin{equation}\label{eqn:squarematrixd}
\begin{aligned}
\Phi_d(t)
&\doteq \lambda \sum_{j=0}^\iy(e_{j+1}-e_{j})(e_{j+1}-e_{j})^T\left(\sum_{i=1}^d\binom{d}{i}[(\pi_d)_j(t)]^i\left(\sum_{m=j+1}^\iy(\pi_d)_m(t)\right)^{d-i}\right)\\
&\qquad+\sum_{j=1}^\iy (e_{j-1}-e_j)(e_{j-1}-e_j)^T(\pi_d)_j(t).
\end{aligned}
\end{equation}
%We claim that $\Phi_d$ in \eqref{eqn:squarematrixd} equals $\Phi$ in \eqref{eqn:squarematrix} where $(L,k)=(d,1)$.
%Indeed, from the definition of $\Del_\ell$ in \eqref{eqn:deldef}, it is clear that for $k=1$, $\Del_\ell = e_{\ell_1+1}-e_{\ell_1}$ and thus
%\begin{equation*}
%\begin{aligned}
%\sum_{\ell\in\Sigma}\Del_\ell\Del_\ell^T\prod_{i=0}^\iy\frac{\pi_i(t)^{\rho_i(\ell)}}{\rho_i(\ell)!}
%&= \sum_{j=0}^\iy(e_{j+1}-e_{j})(e_{j+1}-e_{j})^T\sum_{\{\ell\in\Sigma:\ell_1=j\}}\prod_{i=0}^\iy\frac{\pi_i(t)^{\rho_i(\ell)}}{\rho_i(\ell)!}\\
%&= \sum_{j=0}^\iy(e_{j+1}-e_{j})(e_{j+1}-e_{j})^T\pi_j(t)\sum_{\{\ell\in\Sigma:\ell_1=j\}}\frac{\pi_j(t)^{\rho_j(\ell)-1}}{(\rho_j(\ell)-1)!}\prod_{i=j+1}^\iy\frac{\pi_i(t)^{\rho_i(\ell)}}{\rho_i(\ell)!}.
%\end{aligned}
%\end{equation*}
%Since, for $\ell_1=j$, $\sum_{i=j}^\iy\rho_i(\ell)-1= d-1$, the multinomial theorem gives
%\begin{equation*}
%\frac{\pi_j(t)^{\rho_j(\ell)-1}}{(\rho_j(\ell)-1)!}\prod_{i=j+1}^\iy\frac{\pi_i(t)^{\rho_i(\ell)}}{\rho_i(\ell)!}
%= \frac{1}{(d-1)!}\left(\sum_{i=j}^\iy \pi_i(t)\right)^{d-1}
%\end{equation*}
%from which the claim follows.
Analogous to $G$ in \eqref{eqn:opDFdef}, let $G_d$ be a map from $\ti\bfell_2 \times \cls$ to $\bfell_2$, where $\ti\bfell_2$ is as in \eqref{eqn:tibfelldef}, defined as
\begin{equation}\label{eqn:opDFdefd}
(G_d)_i(x,r)
\doteq \frac{\partial}{\partial u}(F_d)_i(r+ux)\Big|_{u=0}\qquad i\in\NN_0,\ u\in \RR.
\end{equation}
In the special case that $d=2$, this function simply reduces to
\begin{equation*}
(G_2)_i(x,r) = 2\lambda\sum_{m=i}^\iy[x_{i-1}r_m+r_{i-1}x_m-x_ir_{m+1}-r_ix_{m+1}]+(x_{i+1}-x_i).
\end{equation*}
The following result is immediate from Theorem \ref{thm:mainResult}. 
\begin{corollary}\label{cor:pddiff}
	Suppose $\sup_{n\in\NN}\sum_{j=0}^\iy j^2(\pi^n_d)_j(0)<\iy$ and $\pi^n_d(0) \to \pi_0$ in $\cls$ as $n\to \infty$.
Also, suppose  $X_d^n(0)=\sqrt{n}[\pi_d^n(0)-\pi_0]\to x_0$ in probability in $\bfell_2$ and that
\begin{equation*}
\sup_{n\in\NN}\sum_{j=0}^\iy j^2((X_d^n)_j(0))^2<\iy.
\end{equation*} 
Then $X_d^n\Rightarrow X_d$ in $\DD([0,T]:\bfell_2)$ where $X_d$ is the unique weak solution to \eqref{eqn:limitSDE} with values in $\ti\bfell_2$, with $G$ replaced by $G_d$ defined by \eqref{eqn:opDFdefd} and $a(t)$ replaced by $a_d(t)$ which is given as the symmetric square root of the operator $\Phi_d(t)$ in \eqref{eqn:squarematrixd}.
\end{corollary}

\section{Semimartingale Representation}\label{sec:semiMartRep}
In this section we write the state processes using compensated time-changed Poisson processes to give a semimartingale representation for the system. Let $\{N_\ell,\ \ell\in\Sigma\}$ and $\{D_i,\ i\in\NN_0\}$ be collections of mutually independent unit rate Poisson processes. The process $N_\ell$ will be used to represent the stream of jobs requesting files which are stored at servers with queue length configuration (immediately before the time of arrival of the request) $\ell=(\ell_1,\ldots,\ell_L)$. Similarly $D_i$ will represent the stream of jobs completed by servers whose queue length (immediately before the time of completion) is equal to $i$. From the form of the generator in \eqref{eqn:generator} we see that the state process $\pi^n$ can be expressed as,
\begin{equation*}
\pi^n(t)
= \pi^n(0) + \frac{1}{n}\sum_{\ell\in\Sigma}\Del_\ell N_\ell\left(\int_0^t\frac{n\lambda }{\binom{n}{L}}\prod_{i=0}^\iy\binom{n\pi^n_i(s)}{\rho_i(\ell)}ds\right)+\frac{1}{n}\sum_{i=1}^\iy (e_{i-1}-e_i)D_i\left(k\int_0^tn\pi^n_i(s)ds\right).
\end{equation*}

By adding and subtracting the compensators of the Poisson processes one can write the state process as a semimartingale. Namely,  
\begin{equation}\label{eqn:semimartrep}
\pi^n(t)=\pi^n(0)+A^n(t)+M^n(t)
\end{equation}
where
\begin{equation}\label{eqn:Adef}
A^n(t) 
\doteq\sum_{\ell\in\Sigma}\Del_\ell\int_0^t\frac{\lambda}{\binom{n}{L}}\prod_{i=0}^\iy\binom{n\pi^n_i(s)}{\rho_i(\ell)}ds
+k\sum_{i=1}^\iy (e_{i-1}-e_i)\int_0^t\pi^n_i(s)ds
\end{equation}
and
\begin{equation}\label{eqn:MartRep}
\begin{aligned}
M^n(t)&\doteq \sum_{\ell\in\Sigma}\frac{1}{n}\Del_\ell N_\ell\left(\frac{n\lambda}{\binom{n}{L}}\int_0^t\prod_{i=0}^\iy\binom{n\pi^n_i(s)}{\rho_i(\ell)}ds\right)- \sum_{\ell\in\Sigma}\Del_\ell \frac{\lambda}{\binom{n}{L}}\int_0^t\prod_{i=0}^\iy\binom{n\pi^n_i(s)}{\rho_i(\ell)}ds\\
&\qquad+ \sum_{i=1}^\iy \frac{1}{n}(e_{i-1}-e_i) D_i\left(k\int_0^tn\pi^n_i(s)ds\right)- k\sum_{i=1}^\iy (e_{i-1}-e_i) \int_0^t\pi^n_i(s)ds.
\end{aligned}
\end{equation}
From \eqref{eqn:AjzetaExp} and \eqref{eqn:zetaineq}, it follows that for some $c_\zeta\in(0,\iy)$
\begin{equation*}
A_j^n(t) \leq \int_0^t\left(\frac{\lambda}{\binom{n}{L}}c_\zeta n^L[\pi^n_{j-1}(s)+\pi_j^n(s)]+k[\pi_{j+1}^n(s)+\pi^n_j(s)]\right)ds
\end{equation*}
for all $t\in[0,T]$, $n\in\NN$, and $j\in\NN_0$. Thus, there exists a $\kappa\in(0,\iy)$ such that
\begin{equation*}
\begin{aligned}
\sum_{j=0}^\iy A_j^n(t)^2
&\leq \kappa\sum_{j=0}^\iy\int_0^t[\pi_{j-1}^n(s)^2+\pi_{j+1}^n(s)^2+\pi^n_j(s)^2]ds
\leq 3\kappa t
\end{aligned}
\end{equation*}
for all $t\in[0,T]$. Consequently both $M^n(t)$ and $A^n(t)$ take values in $\bfell_2$. A similar argument shows that
$A^n(t)$ in fact takes values in $\bfell_1$.

Similarly, using \eqref{eqn:semimartrep} and \eqref{eq:eq324} for $\pi(t)$, we can express $X^n$ as a semimartingale through the equation
\begin{equation}\label{eqn:semiMart}
X^n(t) = X^n(0) + \bar{A}^n(t) + \bar{M}^n(t)
\end{equation}
where 
\begin{equation}\label{eqn:Abardef}
\bar{A}^n(t) = \sqrt{n}\left[A^n(t)-\int_0^tF(\pi(s))ds\right]
\end{equation}
and $\bar{M}^n(t) = \sqrt{n}M^n(t)$. We note that there is a natural filtration $\{\clf^n_t\}_{0\leq t\leq T}$ on the probability space where the processes $N_\ell,\ D_i$, and $\pi^n$ are defined such that $A^n,\ M^n,\ \pi^n,\ X^n,\ \bar{M}^n,\ \bar{A}^n$ are RCLL processes adapted to the filtration and $M^n,\ \bar{M}^n$ are $\{\clf^n_t\}$-local martingales.

\section{Law of Large Numbers}\label{sec:LLN}
In this section we present the proof of Theorem \ref{thm:LLN}. First, in Section \ref{sec:piTight}, we use the semimartingale representation from Section \ref{sec:semiMartRep} to prove a key tightness property (see Proposition \ref{prop:tightness}). 
Then, in Section \ref{sec:piLimitPoints}, we prove the unique solvability of \eqref{eqn:ODE} and complete the proof of Theorem \ref{thm:LLN} by proving convergence of $\pi^n$ to the unique solution of \eqref{eqn:ODE} in $\CC([0,T]:\cls)$.

\subsection{Tightness}\label{sec:piTight}
In this section we prove tightness of $\{(\pi^n,M^n)\}_{n\in\NN}$. We first recall the notion of $\CC$-tightness.
\begin{definition}
	Let $(\clz, d_{\clz})$ be a Polish space.
For $z\in\DD([0,T]:\clz)$ let $$j_T(z)\doteq\sup_{0\le t\leq T}d_{\clz}(z(t),z(t-)).$$ We say a tight sequence of $\DD([0,T]:\clz)$-valued random variables $\{Z_n\}_{n\in\NN}$ is $\CC$-tight if $j_T(Z_n)\Rightarrow 0$.
\end{definition}
If $Z_n,Z$ are $\DD([0,T]:\clz)$-valued random variables and $Z_n\Rightarrow Z$ then $\PP(Z\in\CC([0,T]:\clz))=1$ if and only if $\{Z_n\}_{n\in\NN}$ is $\CC$-tight \cite{ethier2009markov}.
The following proposition proves the $\CC$-tightness of $\{\pi^n\}_{n\in\NN}$ and convergence of $M^n$ to the zero process. 
\begin{proposition}\label{prop:tightness}
Suppose that $\pi^n(0) \to \pi_0$, in $\cls$, as $n \to \infty$. Then $\{(\pi^n, M^n)\}_{n\in\NN}$ is a $\CC$-tight sequence of
$\DD([0,T]:\cls \times \bfell_2)$-valued random variables. Furthermore, 
 $M^n\Rightarrow 0$ in $\DD([0,T]:\bfell_2)$.
\end{proposition}
\begin{proof}
We first prove the second statement  by arguing that $\E\sup_{0\leq s\leq T}\|M^n(s)\|_2^2\to 0$ as $n\to\iy$. For this, from Doob's inequality, it suffices to show $\E|\lan M^n\ran(T)|\to0$ as $n\to\iy$ where
\begin{equation*}
\lan M^n\ran(s)
\doteq\sum_{j=0}^\iy\lan M^n_j\ran(s),\qquad s\in[0,T].
\end{equation*}
From \eqref{eqn:MartRep} and observing
\begin{equation*}
\sum_{i=1}^\iy \lan e_j,(e_{i-1}-e_i)(e_{i-1}-e_i)^Te_j\ran_2 \pi^n_i(s)=\pi^n_{j+1}(s)+\pi^n_j(s)
\end{equation*}
it follows that
\begin{equation}\label{eqn:Mjs}
\lan M^n_j\ran(t)
=   \frac{\lambda}{n\binom{n}{L}}\int_0^tZ(j,n\pi^n(s))ds+ \frac{k}{n}\int_0^t[\pi^n_{j+1}(s)+\pi^n_j(s)]ds.
\end{equation}
where
\begin{equation}\label{eqn:incoming1}
Z(j,n\pi^n(s))=\sum_{\ell\in\Sigma}\lan e_j,\Del_\ell\Del_\ell^Te_j\ran_2 \prod_{i=0}^\iy\binom{n\pi^n_i(s)}{\rho_i(\ell)}.
\end{equation}
The $\ell$-th term in the sum on the right side of \eqref{eqn:incoming1} is the contribution from jobs that request servers with queue length configuration $\ell$. A fixed $\ell\in\Sigma$ will make non-zero contribution to $\lan e_j,\Del_\ell\Del_\ell^Te_j\ran_2$ if $j$ or $j-1$ is one of the $k$-smallest coordinates in $\ell$. Thus, for a fixed $\ell\in\Sigma$, the $\ell$-th term in \eqref{eqn:incoming1} is nonzero only if $j$ or $j-1$ is a member of the set $(\ell_1,\ldots,\ell_k)$. The contribution from all such $\ell$'s in the sum \eqref{eqn:incoming1} can be counted as follows. Suppose $0\leq i_1\leq k-1$ servers are selected among those with queue length less than $j-1$. This corresponds to $\binom{n\sum_{m=0}^{j-2}\pi_m^n(s)}{i_1}$ different choices of servers. In addition suppose $i_2\leq L-i_1$ and $i_3\leq L-i_1-i_2$ servers are selected among those with queue length equal to $j-1$ and $j$, respectively. This corresponds to $\binom{n\pi^n_{j-1}(s)}{i_2}$ and $\binom{n\pi^n_{j}(s)}{i_3}$ choices, respectively. It follows that $L-i_1-i_2-i_3$ servers must be selected which have queue length larger than $j$ which corresponds to $\binom{n\sum_{m=j+1}^\iy\pi^n_{m}(s)}{L-i_1-i_2-i_3}$ possible choices. Since jobs are only routed to the $k$ shortest servers, 
\begin{equation}\label{eqn:innerEval}
\lan e_j,\Del_\ell\Del_\ell^Te_j\ran_2=[i_2\wedge (k-i_1)-i_3\wedge (k-i_1-i_2)_+]^2.
\end{equation}
It follows that for $x\in n\cls_n$
\begin{equation}\label{eqn:Zdef}
\begin{aligned}
&Z(j,x)
= \\
&\quad\sum_{i_1=0}^{k-1}\binom{\sum_{m=0}^{j-2}x_m}{i_1}\sum_{i_2=0}^{L-i_1}\binom{x_{j-1}}{i_2}\sum_{i_3=0}^{L-i_1-i_2}[i_2\wedge (k-i_1)-i_3\wedge (k-i_1-i_2)_+]^2\binom{x_j}{i_3}\binom{\sum_{m=j+1}^{\iy}x_{m}}{L-i_1-i_2-i_3},
\end{aligned}
\end{equation}
where, recall, we adopt the convention that for $a<b$, $\sum_{i=b}^ax_i=0$.

Note that for non-negative integers $a,b,\ a\geq b$
\begin{equation}\label{eqn:binominequality}
\binom{a}{b}\leq \frac{a^b}{b!}.
\end{equation}
This fact, combined with \eqref{eqn:Zdef} and recalling the fact that $\pi^n(s)\in\cls$ for $s\in[0,T]$, gives the following bound on $Z(j, n\pi^n(s))$
\begin{equation}\label{eqn:Zineq}
\begin{aligned}
Z(j,n\pi^n(s))
&\leq \sum_{i_1=0}^{k-1}\frac{(n\sum_{m=0}^{j-2}\pi^n_m(s))^{i_1}}{i_1!}\sum_{i_2=0}^{L-i_1}\frac{(n\pi^n_{j-1}(s))^{i_2}}{i_2!}\\
&\qquad\times\sum_{i_3=0}^{L-i_1-i_2}k^21_{\{i_2\vee i_3 > 0\}}\frac{(n\pi^n_j(s))^{i_3}}{i_3!}\frac{(n\sum_{m=j+1}^{\iy}\pi^n_{m}(s))^{L-i_1-i_2-i_3}}{(L-i_1-i_2-i_3)!}\\
&\leq n^L\sum_{i_1=0}^{k-1}\sum_{i_2=0}^{L-i_1}\sum_{i_3=0}^{L-i_1-i_2}k^21_{\{i_2\vee i_3 > 0\}}(\pi^n_{j-1}(s))^{i_2}(\pi^n_j(s))^{i_3}\\
&\leq c_Zn^L(\pi^n_{j-1}(s)+\pi^n_j(s)).
\end{aligned}
\end{equation}
for some $c_Z\in(0,\iy)$. Using \eqref{eqn:Zineq} in \eqref{eqn:Mjs} gives
\begin{equation}\label{eqn:quadMRate}
\begin{aligned}
\E|\lan M^n\ran(t)|
&\leq\E\left| \frac{2\lambda(n-L)!L!c_Zn^L}{n\times n!}\int_0^t\sum_{j=0}^\iy\pi^n_j(s)ds\right|+\E\left|\frac{2k}{n}\int_0^t\sum_{j=0}^\iy\pi^n_j(s)ds\right| \\
&\leq \left| \frac{2\lambda(n-L)!L!c_Zn^L}{n\times n!}t\right|+\left|\frac{2k}{n}t\right|.
\end{aligned}
\end{equation}
Thus $\E|\lan M^n\ran_T|\to0$ and consequently $\E\sup_{0\leq s\leq T}\|M^n(s)\|_2^2\to0$ as $n\to\iy$. It follows that $M^n\Rightarrow 0$ in $\DD([0,T]:\bfell_2)$ which completes the proof of $(ii)$.

The tightness of $\{\pi^n\}_{n\in\NN}$ in $\DD([0,T]:\cls)$ follows as in the proof of Theorem 3.4 of \cite{graham2000chaoticity}. Namely, it suffices to show tightness of $\{Q_1^n\}_{n\in\NN}$ in $\DD([0,T]:\NN)$ (cf. \cite{sznitman1991topics}). However, this tightness is an immediate consequence of the fact that the jumps of $Q_1^n$ can be embedded in a Poisson process with rate $\lambda L+k$.

Finally in order to show that $\{\pi^n\}_{n\in\NN}$ is $\CC$-tight it suffices to show that 
\begin{equation*}
j_T(\pi^n)
\doteq\sup_{0\leq t\leq T}d_0(\pi^n(t),\pi^n(t-))
\to0\text{ as }n\to\iy.
\end{equation*}
There are two types of jumps, those corresponding to incoming jobs and those corresponding to jobs being processed. When a job arrives in the system, the dispatcher assigns it to $k$ different servers causing the queue length of each of the $k$ chosen servers to increase by one. It follows that the jump size of such an event can be bounded by $\frac{2k}{n}$. When a job is processed, the corresponding queue length will drop by 1 and so the jump size of such an event can be bounded by $\frac{2}{n}$. Therefore
$
j_T(\pi^n)
\leq \frac{2+2k}{n}
\to 0$
which completes the proof.
\end{proof}

\subsection{Convergence}\label{sec:piLimitPoints}
In this section we provide the proof of Theorem \ref{thm:LLN}. Since we have already proved tightness of $\{\pi^n\}_{n\in\NN}$ in Section \ref{sec:piTight}, all that remains is to prove uniqueness of solutions of \eqref{eqn:ODE} in an appropriate class and to characterize the limit of any weakly convergent subsequence as the unique solution to \eqref{eqn:ODE}. We first present the following Lipschitz property for the map $F:\cls\to\bfell_1$, defined in \eqref{eqn:codingF}, that will give uniqueness of the solutions to \eqref{eqn:ODE}. We remark that in the proof of Theorem \ref{thm:mainResult} we will need a stronger Lipschitz property of $F$ in the $\bfell_2$ norm. This Lipschitz property is not immediate on the space $\cls$ but, as shown in Lemma \ref{lem:lipschitz2}, is satisfied on a smaller class $\clv_M$.
\begin{lemma}\label{lem:lipschitz}
The map $F$ is a Lipschitz function from $\cls$ to $\bfell_1$. Namely, there exists an $C_1\in(0,\iy)$ such that for any $r,\ti r\in\cls$, 
\begin{equation}\label{eqn:LipschitzDisplay1}
\|F(r)-F(\ti r)\|_1\leq C_1\|r-\ti r\|_1.
\end{equation}
\end{lemma}
\begin{proof}
Let $r,\ti r\in\cls$ and, for  $i_1\in\NN_0$ and $j,i_2\in\NN$, define $R_{j,i_1,i_2}(r,\ti r)$ as
\begin{equation}\label{eqn:Rdef}
R_{j,i_1,i_2}(r,\ti r)
\doteq\left(\sum_{m=0}^{j-1}r_m\right)^{i_1}\left(\sum_{m=j+1}^{\iy}r_{m}\right)^{L-i_1-i_2}r^{i_2}_j- \left(\sum_{m=0}^{j-1}\ti r_m\right)^{i_1}\left(\sum_{m=j+1}^{\iy}\ti r_{m}\right)^{L-i_1-i_2}\ti r_j^{i_2}.
\end{equation}
Note that for any $a,b,c,\ti a, \ti b,\ti c\in\RR_+$,
\begin{equation}\label{eqn:tripleIneq}
abc-\ti a \ti b \ti c= ab (c-\ti c)+a(b-\ti b)\ti c+(a-\ti a)\ti b\ti c.
\end{equation}
Combining \eqref{eqn:Rdef}, \eqref{eqn:tripleIneq}, and the fact that $r,\ti r\in\cls$, we have
\begin{equation}\label{eqn:Rbound1}
\begin{aligned}
|R_{j,i_1,i_2}(r,\ti r)|
&\leq |r_{j}^{i_2}-\ti r_{j}^{i_2}| +\ti r_{j}^{i_2}\left|\left(\sum_{m=j+1}^\iy r_m\right)^{L-i_1-i_2}-\left(\sum_{m=j+1}^\iy \ti r_m\right)^{L-i_1-i_2}\right|\\
&\qquad+\ti r_{j}^{i_2}\left|\left(\sum_{m=0}^{j-1} r_m\right)^{i_1}-\left(\sum_{m=0}^{j-1} \ti r_m\right)^{i_1}\right|.
\end{aligned}
\end{equation}
For any $a,b\in\RR$ and $i\in\NN$, $(a^i-b^i)= (a-b)\sum_{j=1}^{i}a^{i-j}b^{j-1}$. Thus, if $a,b\in[0,1]$ and $i\leq L$, $|a^i-b^i|\leq|a-b|L$.
This inequality along with \eqref{eqn:Rbound1} implies there exist $\kappa_1,\kappa_1'>0$ such that for all $i_1,i_2\leq L,\ i_2>0,$
\begin{equation}\label{eqn:Rineq}
\begin{aligned}
 |R_{j,i_1,i_2}(r,\ti r)|
&\leq \kappa_1'\left(|r_{j}-\ti r_{j}| +\ti r^{i_2}_{j}\sum_{m=j+1}^\iy|r_m-\ti r_m|+\ti r^{i_2}_{j}\sum_{m=0}^{j-1}|r_m-\ti r_m|\right)\\
&\leq \kappa_1(|r_{j}-\ti r_{j}|+\ti r_{j}\|r-\ti r\|_1).
\end{aligned}
\end{equation}
The definition of $F$ (see \eqref{eqn:codingF}) and the triangle inequality imply,
\begin{equation}\label{eqn:Ftriangle}
\|F(r)-F(\ti r)\|_1
\leq  \lambda L!\sum_{j=0}^\iy |\bar{\zeta}^\del(j,r)-\bar{\zeta}^\del(j,\ti r)| + k\sum_{j=0}^\iy|(r-\ti r)_{j+1}-(r-\ti r)_j|.
\end{equation}
Noting that
\begin{equation*}
\bar{\zeta}^\del(j,r)-\bar{\zeta}^\del(j,\ti r)
= [\bar{\zeta}(j-1,r)-\bar{\zeta}(j-1,\ti r)]- [\bar{\zeta}(j,r)-\bar{\zeta}(j,\ti r)],
\end{equation*}
it follows that
\begin{equation}\label{eqn:zetatoR}
\begin{aligned}
\sum_{j=0}^\iy |\bar{\zeta}^\del(j,r)-\bar{\zeta}^\del(j,\ti r)|
&\leq 2\sum_{j=0}^\iy |\bar{\zeta}(j,r)-\bar{\zeta}(j,\ti r)|\leq \kappa_2\sum_{j=0}^\iy\sum_{i_1=0}^{k-1}\sum_{i_2=1}^{L-i_1}|R_{j,i_1,i_2}(r,\ti r)|
\end{aligned}
\end{equation}
where the second inequality follows from the the definitions of $\bar{\zeta}$ and $R$. Combining \eqref{eqn:zetatoR} with \eqref{eqn:Ftriangle} and applying \eqref{eqn:Rineq} yields, for some $\kappa_3>0$,
\begin{equation*}
\begin{aligned}
\|F(r)-F(\ti r)\|_1
&\leq  \kappa_2\lambda L!\sum_{j=0}^\iy\sum_{i_1=0}^{k-1}\sum_{i_2=1}^{L-i_1}|R_{j,i_1,i_2}(r,\ti r)|+ 2k\sum_{j=0}^\iy|r_j-\ti r_j|\\
&\leq \kappa_3\sum_{j=0}^\iy\left[|r_{j}-\ti r_{j}| +\ti r_{j}\|r-\ti r\|_1\right]+ 2k\|r-\ti r\|_1
\end{aligned}
\end{equation*}
and thus with $C_1\doteq 2(\kappa_3+k)$, \eqref{eqn:LipschitzDisplay1} is satisfied for all $r,\ti r\in \mathcal{S}$ which proves the result.
\end{proof}

Using the above Lipschitz property of $F$ we can now complete the proof of Proposition \ref{prop:uniqueness}.

\begin{proof}[Proof of Proposition \ref{prop:uniqueness}]
Existence of a $\pi\in \CC([0,T]:\cls)$ that solves \eqref{eqn:ODE} will be shown below in the proof of Theorem \ref{thm:LLN}. We now argue uniqueness. Suppose $\pi$ and $\ti\pi$ are two elements of $\CC([0,T]:\cls)$ satisfying \eqref{eqn:ODE} with $\pi(0)=\ti\pi(0)=\pi_0$.  The Lipschitz property of $F$ proved in Lemma \ref{lem:lipschitz} implies, for all $t\in[0,T]$
\begin{equation*}
\begin{aligned}
\|\pi(t)-\ti\pi(t)\|_1
&= \left\|\int_0^t[F(\pi(s))-F(\ti\pi(s))]ds\right\|_1
\leq\int_0^t\|F(\pi(s))-F(\ti\pi(s))\|_1ds\\
&\leq C_1\int_0^t\|\pi(s)-\ti\pi(s)\|_1ds.
\end{aligned}
\end{equation*}
The result follows.
\end{proof}

We now proceed to the proof of Theorem \ref{thm:LLN}.
\begin{proof}[Proof of Theorem \ref{thm:LLN}]
From Proposition \ref{prop:tightness} we have that $\{\pi^n\}_{n\in\NN}$ is a $\CC$-tight sequence of $\DD([0,T]:\cls)$-valued random variables.

Note from \eqref{eqn:semimartrep} that for all $j\in\NN_0$,
\begin{equation}\label{eqn:piexpansion}
\pi^n(t) = \pi(0)+V^n(t)+M^n(t)+\int_0^tF(\pi^n(s))ds
\end{equation}
where
\begin{equation*}
V^n(t) \doteq A^n(t) - \int_0^tF(\pi^n(s))ds.
\end{equation*}
From the definition of $A^n$ in \eqref{eqn:Adef} we see that
	\begin{equation}\label{eqn:Ajdef}
	A^n_j(t) = \int_0^t\left(\sum_{\ell\in\Sigma}\lan\Del_\ell, e_j\ran_2\frac{\lambda}{\binom{n}{L}}\prod_{i=0}^\iy\binom{n\pi^n_i(s)}{\rho_i(\ell)}
	+k [\pi_{j+1}^n(s)-\pi_{j}^n(s)]\right)ds.
	\end{equation} 
	By a similar argument used to obtain the representation in \eqref{eqn:Zdef},
	\begin{equation}\label{eqn:jexpansion5}
	\sum_{\ell\in\Sigma}\lan\Del_\ell, e_j\ran_2\prod_{i=0}^\iy\binom{n\pi^n_i(s)}{\rho_i(\ell)}
	= [\zeta(j-1,n\pi^n(s))-\zeta(j,n\pi^n(s))]
	\end{equation}
	where for $x\in n\cls_n$
	\begin{equation}\label{eqn:zetadef}
	\zeta(j,x)
	\doteq \sum_{i_1=0}^{k-1}\binom{\sum_{m=0}^{j-1}x_m}{i_1}\sum_{i_2=1}^{L-i_1}[i_2\wedge(k-i_1)]\binom{x_j}{i_2}\binom{\sum_{m=j+1}^{\iy}x_{m}}{L-i_1-i_2}.
	\end{equation}
	From our convention that $x_{-1}=0$, we see that $\zeta(-1,x)=0$. 
	In addition, recalling the conventions that for $a<b$, $\sum_{i=b}^a x_i = 0$ and that $\binom{0}{0}=1$ we see $\zeta(0,x)$ is well defined.
	Combining \eqref{eqn:Ajdef}, \eqref{eqn:jexpansion5}, and \eqref{eqn:zetadef} gives the following representation for $A^n_j$
	\begin{equation}\label{eqn:AjzetaExp}
	A^n_j(t) = \frac{\lambda}{\binom{n}{L}}\int_0^t[\zeta(j-1,n\pi^n(s))-\zeta(j,n\pi^n(s))]ds
	+k\int_0^t [\pi_{j+1}^n(s)-\pi_{j}^n(s)]ds.
	\end{equation}
For each fixed $j, i_1\in\NN_0$ and $i_2\in\NN$ with $i_1,i_2\leq L$ we have
\begin{equation}\label{eqn:binomToExp}
\begin{aligned}
&\binom{n\sum_{m=0}^{j-1}\pi^n_m(s)}{i_1}[i_2\wedge(k-i_1)]\binom{n\pi^n_j(s)}{i_2}\binom{n\sum_{m=j+1}^{\iy}\pi^n_{m}(s)}{L-i_1-i_2}\\
&\qquad= n^L\frac{\left(\sum_{m=0}^{j-1}\pi^n_m(s)\right)^{i_1}}{i_1!}[i_2\wedge(k-i_1)]\frac{(\pi^n_j(s))^{i_2}}{i_2!}\frac{\left(\sum_{m=j+1}^{\iy}\pi^n_{m}(s)\right)^{L-i_1-i_2}}{(L-i_1-i_2)!}+\hat{R}_n(j,i_1,i_2,s)
\end{aligned}
\end{equation}
where
\begin{equation*}
\sup_{i_1,i_2\leq L}|\hat{R}_n(j,i_1,i_2,s)|\leq \kappa_1n^{L-1}\pi^n_j(s)
\end{equation*}
and thus, from the definition of $\zeta$ and $\bar{\zeta}$ in \eqref{eqn:zetadef} and \eqref{eqn:zetabardef},
\begin{equation}\label{eqn:zetatobar}
\left|\zeta(j,n\pi^n(s))- \frac{n!}{(n-L)!}\bar{\zeta}(j,\pi^n(s))\right|\leq \kappa_2n^{L-1}\pi^n_j(s)\ \forall\ s\in[0,T].
\end{equation}
Furthermore, using the definition of $A^n$ in \eqref{eqn:AjzetaExp} and $F$ in \eqref{eqn:codingF}, \eqref{eqn:zetatobar} implies
\begin{equation}\label{eqn:AtoF}
\sup_{0\leq t\leq T}\|V^n(t)\|_2=\sup_{0\leq t\leq T}\left\|A^n(t)- \int_0^tF(\pi^n(s))ds\right\|_2\leq \frac{\kappa_3}{n}.
\end{equation}
Also from Proposition \ref{prop:tightness}, $M^n\Rightarrow0$ in $\DD([0,T]:\bfell_2)$. Combining these observations with the tightness of $\pi^n$, we have subsequential convergence of $(\pi^n,M^n,V^n)$ to $(\pi,0,0)$, in distribution, in $\DD([0,T]:\cls\times\bfell_2\times\bfell_2)$ for some $\CC([0,T]:\cls)$-valued $\pi$.
By appealing to the Skorohod representation theorem we can assume that this convergence holds a.s.
Noting that $r\mapsto F_j(r)$ is a continuous map from $\cls$ to $\RR$ for each $j\in\NN_0$ we have that $F_j(\pi^n(s))\to F_j(\pi(s))$ as $n\to\iy$ for all $j\in\NN_0$ and $s\in[0,T]$.
Thus, upon sending $n\to\iy$ in \eqref{eqn:piexpansion}, \eqref{eq:eq305} and the dominated convergence theorem imply that almost surely,
\begin{equation*}
\pi_j(t) = (\pi_0)_j+\int_0^tF_j(\pi(s))ds,\text{ for all }t\in[0,T],\ j\in\NN_0.
\end{equation*}
This shows that $\pi$ satisfies \eqref{eqn:ODE}. The result now follows from the uniqueness property shown in Proposition \ref{prop:uniqueness}.
\end{proof}

\section{Diffusion Approximation}\label{sec:diffApprox}
In this section we prove Theorem \ref{thm:mainResult}. Section \ref{sec:mombds} presents some moment estimates on $\pi^n$ which will be used in the proof of Theorem \ref{thm:mainResult}. Section \ref{sec:diffTightness} then proves tightness of the sequence of centered and scaled state processes $\{X^n\}_{n\in\NN}$. Section \ref{sec:diffConv} completes the proof of Theorem \ref{thm:mainResult} by proving unique solvability of the SDE \eqref{eqn:limitSDE} (Theorem \ref{prop:DiffUniqueness}) and characterizing limit points of $X^n$ as this unique solution.

\subsection{Moment Bounds}\label{sec:mombds}
The following elementary lemma will be useful in the proof of Lemma \ref{lem:mombd}.
\begin{lemma}\label{lem:tailbound}
For all $t\geq 0$, $k\in\NN$, and $n\in\NN$, $\lim_{m\to\iy}\E m^k\sup_{0\leq s\leq t}\pi^n_m(s)=0$.
\end{lemma}
\begin{proof}
Fix $n\in\NN$. Note that file requests arrive at rate $n\lambda$. Let $N$ be a Poisson process representing the total flow of such file requests. Also let $m^*=\sup\{m:\pi_m^n(0)>0\}$ be the length of the largest queue at time $0$. Note that since the system consists of $n$ queues, $m^*$ must be finite for any fixed $n$. Then for $m>m^*$,
\begin{equation*}
\begin{aligned}
\E m^k\sup_{0\leq s\leq t}\pi^n_m(s)
&= \E\sup_{0\leq s\leq t}1_{\{N(t)\geq m-m^*\}} m^k\pi^n_m(s)+\E\sup_{0\leq s\leq t}1_{\{N(t)<m-m^*\}}m^k\pi^n_m(s)\\
& \leq m^k\PP(N(t)\geq m-m^*).
\end{aligned}
\end{equation*}
Thus, from Markov's inequality, for $m>m^*$
\begin{equation*}
\E m^k\sup_{0\leq s\leq t}\pi^n_m(s)
\leq m^ke^{-(m-m^*)}e^{n\lambda t(e-1)}.
\end{equation*}
The result follows.
\end{proof}
In the next lemma we will we establish two key moment bounds that will be needed in the tightness proof (see proof of Proposition \ref{prop:diffTightness}). 
\begin{lemma}\label{lem:mombd}
	Suppose $\sup_{n\in\NN}\sum_{j=0}^\iy j^2\pi^n_j(0) \doteq c_{\pi(0)}< \infty$. Then 
	\begin{equation}\label{eqn:probmzrtight}
	\sup_{n\in\NN}\E\sup_{0\leq t\leq T}\left(\sum_{j=0}^\iy j\pi_j^n(t)\right)^2<\iy
	\end{equation}
	and 
	\begin{equation}\label{eqn:secmomentbound}
	\sup_{n\in\NN}\E\int_0^T\sum_{j=0}^\iy j^2\pi^n_j(t)dt<\iy.
	\end{equation}
\end{lemma}
\begin{proof}
	Since $\pi^n(t)=\pi^n(0)+A^n(t)+M^n(t)$, we can write for fixed $K \in \NN$
	\begin{equation}\label{eqn:firstmomExpan}
	\begin{aligned}
	\E\sup_{0\leq t\leq T}\left|\sum_{j=0}^K j\pi_j^n(t)\right|^2
	&\leq 3\left|\sum_{j=0}^K j\pi_j^n(0)\right|^2+3\E\sup_{0\leq t\leq T}\left|\sum_{j=0}^K jA^n_j(t)\right|^2+3\E\sup_{0\leq t\leq T}\left|\sum_{j=0}^K jM^n_j(t)\right|^2.
	\end{aligned}
	\end{equation}
	Using \eqref{eqn:jexpansion5}, for $K\in\NN$, we can write
	\begin{equation}\label{eqn:jexpansion3}
	\begin{aligned}
	\sum_{j=0}^K j\sum_{\ell\in\Sigma}\lan\Del_\ell, e_j\ran_2\prod_{i=0}^\iy\binom{n\pi^n_i(s)}{\rho_i(\ell)}
	&=\sum_{j=1}^K j\left[\zeta(j-1,n\pi^n(s))-\zeta(j,n\pi^n(s))\right]\\
	&=\sum_{j=0}^{K-1} \zeta(j,n\pi^n(s))-K\zeta(K,n\pi^n(s))
	\end{aligned}
	\end{equation}
	and
	\begin{equation}\label{eqn:jexpansion4}
	k\sum_{j=0}^K j[\pi_{j+1}^n(s)-\pi_{j}^n(s)]
	= -k\left(\sum_{j=1}^K\pi_j^n(s)-K\pi_{K+1}^n(s)\right).
	\end{equation}
	Using similar bounds as in \eqref{eqn:Zineq}, for some $c_\zeta\in(0,\iy)$
	\begin{equation}\label{eqn:zetaineq}
	\zeta(j,n\pi^n(s))\leq c_{\zeta}n^L \pi_j^n(s).
	\end{equation}
	%Using the above bound and Lemma \ref{lem:tailbound} we also have that $K\pi_{K+1}^n(s)\to 0$ and $K\zeta(K,n\pi^n(s))\to0$ as $K\to\iy$.
	%This combined with \eqref{eqn:jexpansion3} and \eqref{eqn:jexpansion4} 
	The above bound implies that for some $\kappa_1\in (0,\infty)$, for all $n, K \in \NN$
	\begin{equation*}
	\E\sup_{0\leq t\leq T}\left[\frac{\lambda }{\binom{n}{L}}\int_0^t\sum_{j=1}^K\zeta(j-1,n\pi^n(s))+k\int_0^t\sum_{j=0}^K\pi_j^n(s)ds\right]^2\leq \E\left[\left(c_\zeta n^L\frac{\lambda}{\binom{n}{L}}+k\right)T\right]^2
	\leq \kappa_1.
	\end{equation*}
	Combined with \eqref{eqn:AjzetaExp}, \eqref{eqn:jexpansion3}, and \eqref{eqn:jexpansion4}, the above estimate gives, for all $n, K \in \NN$,
	\begin{equation}
		\label{eq:eq1030}
		\E\sup_{0\leq t\leq T}\left|\sum_{j=0}^K jA_j^n(t)\right|^2 \le
		\kappa_2\left(1+ K \E\left[\sup_{0\leq t\leq T} (\pi^n_{K}(t) + \pi^n_{K+1}(t))\right]\right).
	\end{equation}

	We now consider $\E\sup_{0\leq t\leq T}|\sum_{j=0}^K jM^n_j(t)|^2$. Since $\sum_{j=0}^K jM^n_j(t)$ is a martingale, Doob's inequality implies that
	\begin{equation}\label{eqn:BGD}
	\begin{aligned}
	\E\sup_{0\leq t\leq T}\left|\sum_{j=0}^K jM^n_j(t)\right|^2
	&\leq 
	4\E\left\lan\sum_{j=0}^K  jM^n_j\right\ran(T)
	=  4\E\sum_{j_1=0}^K\sum_{j_2=0}^K j_1j_2\lan M^n_{j_1},M^n_{j_2}\ran(T).
	\end{aligned}
	\end{equation}
	The diagonal terms ($j_1=j_2$) in the above sum are given by \eqref{eqn:Mjs}. We now consider the off-diagonal terms. Fix $0\le j_1<j_2\le K$ and note that in order to compute $\lan M^n_{j_1},M^n_{j_2}\ran(T)$ we must expand
	% 
	% 
	% To derive the off-diagonal terms, notice that for $\ell\in\Sigma$, $N_\ell$ contributes to $\lan M^n_{j_1},M^n_{j_2}\ran$ if $\{j_1,j_2\}\in\{\ell_1,\ldots,\ell_k\}$ (i.e. of $j_1$ and $j_2$ are two of the $k$ smallest elements of $\ell$).  Fixing $j_1<j_2$, we write
	\begin{align}\label{eqn:incoming2}
	Z(j_1,j_2,n\pi^n(s))\doteq\sum_{\ell\in\Sigma} \lan e_{j_1},\Del_\ell \Del_\ell^T e_{j_2}\ran_2\prod_{i=0}^\iy\binom{n\pi^n_i(s)}{\rho_i(\ell)}.
	\end{align}
	Similar to \eqref{eqn:incoming1}, the $\ell$-th term in \eqref{eqn:incoming2} is the contribution from jobs that request servers with queue length configuration $\ell$. 
	A fixed $\ell\in\Sigma$ will make non-zero contribution to $\lan e_{j_1},\Del_\ell\Del_\ell^Te_{j_2}\ran_2$ if ($j_1$ or $j_1-1$) \emph{and} ($j_2$ or $j_2-1$) are among the $k$-smallest coordinates in $\ell$. 
	I.e. for a fixed $\ell\in\Sigma$, the $\ell$-th term is nonzero only if ($j_1$ or $j_1-1$) is a member of the set $(\ell_1,\ldots,\ell_k)$ and ($j_2$ or $j_2-1$) is also a member. The contribution from all such $\ell$'s in the sum \eqref{eqn:incoming2} can be counted in a method analogous to the one used to obtain \eqref{eqn:Zdef}. Namely, we count the number of choices of servers with queue length less than $j_1-1$, equal to $j_1-1$, equal to $j_1$, between $j_1$ and $j_2-1$, equal to $j_2-1$, equal to $j_2$, and larger than $j_2$. One must be careful in the cases $j_2 -1= j_1$ and $j_2-1 = j_1+1$. In both cases there are no servers with length between $j_1$ and $j_2-1$. 
	In the first case above ($j_2-1=j_1$), we must also be careful not to double count. To ensure this we include an indicator function $1_{\{j_2>j_1+1\}}$ in the upper index of the binomial coefficient corresponding to the selection of servers with queue length equal to $j_2-1$.
	Combining these observations we see that for $x\in n\cls_n$,
	\begin{equation}\label{eqn:Zdef4}
	\begin{aligned}
	Z(j_1,j_2,x)&=\sum_{\ell\in\Sigma}\lan e_{j_1},\Del_\ell\Del_\ell^Te_{j_2}\ran_2 \prod_{i=0}^\iy\binom{x_i}{\rho_i(\ell)}\\
	&=\sum_{i_1=0}^{k-2}\binom{\sum_{m=0}^{j_1-2}x_m}{i_1}\sum_{i_2=0}^{k-i_1-1}\binom{x_{j_1-1}}{i_2}\sum_{i_3=0}^{k-i_1-i_2-1}[i_2-i_3]\binom{x_{j_1}}{i_3}\\
	&\qquad\times\sum_{i_4=0}^{k-i_1-i_2-i_3-1}
	\binom{\sum_{m=j_1+1}^{j_2-2}x_m}{i_4}\sum_{i_5=0}^{L-\sum_{n=1}^4i_n}\binom{x_{j_2-1}1_{\{j_2>j_1+1\}}}{i_5}\\
	&\qquad\times\sum_{i_6=0}^{L-\sum_{n=1}^5i_n}\left[(1_{\{j_2=j_1+1\}}(i_3-i_5)+i_5)\wedge\left(k-\sum_{n=1}^4i_n\right)_+-i_6\wedge\left(k-\sum_{n=1}^5i_n\right)_+\right]\\
	&\qquad\times\binom{x_{j_2}}{i_6}\binom{\sum_{m=j_2+1}^\iy x_m}{L-\sum_{n=1}^6i_n}.
	\end{aligned}
	\end{equation}
	For $j_1>j_2$ we define $Z(j_1,j_2,x)\doteq Z(j_2,j_1,x)$. The contribution to $\lan M^n_{j_1},M^n_{j_2}\ran(T)$, for $j_1\neq j_2$, from completed jobs is similarly given as
	\begin{equation}\label{eqn:outgoing2}
	\begin{aligned}
	&\sum_{i=1}^\iy\lan e_{j_1},(e_{i-1}-e_i)(e_{i-1}-e_i)^Te_{j_2}\ran_2\pi_i^n(s)
	=-1_{\{j_1=j_2-1\}}\pi^n_{j_2}(s)-1_{\{j_1-1=j_2\}}\pi^n_{j_1}(s).
	\end{aligned}
	\end{equation}
	Combining \eqref{eqn:Zdef4} and \eqref{eqn:outgoing2} gives, for $j_1,j_2\in\NN_0$,
	\begin{align}\label{eqn:Zdef2}
	\begin{aligned}
	\lan M^n_{j_1},M^n_{j_2}\ran(T)
	&=\frac{\lambda}{n\binom{n}{L}}\int_0^TZ(j_1,j_2,n\pi^n(s))ds\\
	&\qquad + \frac{k}{n}\int_0^T [1_{\{j_1=j_2\}}[\pi^n_{j_1}(s)+\pi^n_{j_1+1}]-1_{\{j_1=j_2-1\}}\pi^n_{j_2}(s)-1_{\{j_1-1=j_2\}}\pi^n_{j_1}(s)]ds,
	\end{aligned}
	\end{align}
	where, by convention, $Z(j,j,x)\doteq Z(j,x)$. Referring to the definition of $Z$ in \eqref{eqn:Zdef4}, note that 
	% the indices $i_2,i_3,i_5,$ and $i_6$ correspond to the lower index in the binomial coefficients with $x_{j_1-1},x_{j_1},x_{j_2-1}$, and $x_{j_2}$ as the upper indices, respectively.
	% 	Furthermore, notice that, 
	for $j_2>j_1+1$, $Z(j_1,j_2,x)=0$ unless ($i_2$ or $i_3$) are greater than zero and ($i_5$ or $i_6$) are greater than zero. 
	In the case that $j_2=j_1+1$, $Z(j_1,j_2,x)=0$ unless ($i_2$ or $i_3$) are greater than zero and ($i_3$ or $i_6$) are greater than zero. 
	Therefore \eqref{eqn:binominequality} implies there exists a $\ti c_Z\in(0,\iy)$ such that for $r\in \cls_n$ and $j_1< j_2$,
	\begin{equation}\label{eqn:offDiagBound}
	Z(j_1,j_2,nr)\leq\ti c_Zn^L[r_{j_1}r_{j_2}+r_{j_1-1}r_{j_2}+r_{j_1}r_{j_2-1}+r_{j_1-1}r_{j_2-1}+1_{\{j_2=j_1+1\}}r_{j_1}].
	\end{equation}
	Combining this with \eqref{eqn:Zineq} and \eqref{eqn:Zdef2}, we have for some $\kappa_3', \kappa_3 \in (0,\infty)$ and all $n, K \in \NN$
	\begin{equation}\label{eqn:crossquad}
	\begin{aligned}
	&\sum_{j_1=0}^K\sum_{j_2=0}^K j_1j_2\lan M^n_{j_1},M^n_{j_2}\ran(T)\\
	&\qquad\le
	\frac{\kappa_3'}{n}\left[\int_0^T\sum_{j_1=0}^K\sum_{j_2=0}^K (j_1+1)(j_2+1)\pi^n_{j_1}(t)\pi^n_{j_2}(t)dt+\int_0^T\sum_{j=1}^{K+1} j(j+1)\pi^n_j(t)dt\right]\\
	&\qquad\leq\frac{\kappa_3}{n}\left[\int_0^T\left(\sum_{j=0}^{K} j^2\pi^n_j(t)+(K+1)^2\pi^n_{K+1}(t)\right)dt+1\right].	
	\end{aligned}
	\end{equation}
	Recalling $\pi^n(t) = \pi^n(0)+A^n(t)+M^n(t)$ we have that for all $K, n \in \NN$ 
	\begin{equation*}
	\begin{aligned}
	\E\int_0^T\sum_{j=0}^K j^2\pi^n_j(t)dt
	&= \int_0^T\sum_{j=0}^K j^2\pi^n_j(0)dt+\E\int_0^T\sum_{j=0}^K j^2A^n_j(t)dt+ \int_0^T\E\sum_{j=0}^K j^2M^n_j(t)dt\\
	&\le \E\int_0^T\sum_{j=0}^K j^2A^n_j(t)dt + \kappa_4,
	\end{aligned}
	\end{equation*}
	where $\kappa_4=c_\pi(0)T$ and the last inequality follows on using the fact that $M^n_j(t)$ is a martingale.
	Thus, from \eqref{eqn:AjzetaExp}, for some $\kappa_5 \in (0,\infty)$ and all $K, n \in \NN$ 
	\begin{equation}\label{eqn:piSecMoment}
	\begin{aligned}
	\E\int_0^T\sum_{j=0}^K j^2\pi^n_j(t)dt
	&\le \frac{\kappa_5}{n^L}\E\int_0^T\sum_{j=1}^K j^2\int_0^t[\zeta(j-1,n\pi^n(s))-\zeta(j,n\pi^n(s))]ds\,dt\\
	&\qquad+\kappa_5\E\int_0^T\sum_{j=1}^K j^2\int_0^t[\pi^n_{j+1}(s)-\pi^n_j(s)]ds\,dt+\kappa_5.
	\end{aligned}
	\end{equation}
	Using the fact that for any $a_0,\ldots,a_K\in\RR$
	\begin{equation*}
	\sum_{j=1}^Kj^2[a_{j-1}-a_j]
	= \sum_{j=1}^K[(j-1)^2a_{j-1}-j^2a_j+(2j-1)a_{j-1}]
	= -K^2a_K+\sum_{j=1}^K(2j-1)a_{j-1}
	\end{equation*}
	and 
	\begin{equation*}
	\sum_{j=0}^Kj^2[a_{j+1}-a_j]
	= \sum_{j=0}^K[(j+1)^2a_{j+1}-j^2a_j-(2j+1)a_{j+1}]
	= (K+1)^2a_{K+1}-\sum_{j=0}^K(2j+1)a_{j+1}
	\end{equation*}
	in \eqref{eqn:piSecMoment} we have that, for some $\kappa_6 \in (0,\infty)$ and all $K, n \in \NN$ 
	\begin{equation}\label{eqn:piSecMoment2}
	\begin{aligned}
	\E\int_0^T\sum_{j=0}^K j^2\pi^n_j(t)dt
	&\leq\frac{\kappa_5}{n^L}\E\int_0^T\int_0^t\sum_{j=0}^K (2j-1)\zeta(j-1,n\pi^n(s))ds\,dt\\
	&\quad+\kappa_5\E\int_0^T\int_0^t (K+1)^2\pi^n_{K+1}(s)ds\,dt+\kappa_5\\
	&\leq \kappa_6\E\int_0^T [K^2\sup_{0\le s \le t}\pi^n_{K+1}(s)+\sup_{0\le s \le t}\sum_{j=0}^K j\pi^n_{j}(s)] dt+\kappa_6
	\end{aligned}
	\end{equation}
	where the second inequality follows from \eqref{eqn:zetaineq}. 
	Thus it follows from \eqref{eqn:BGD} and \eqref{eqn:crossquad} that for some $\kappa_7\in (0,\infty)$
	\begin{align}
	\E\sup_{0\leq t\leq T}\left|\sum_{j=0}^K jM^n(t)\right|^2
	&\leq \frac{\kappa_3}{n}\left[\int_0^T\E\sum_{j=0}^K j^2\pi^n_j(t)dt+\gamma^n_{K}T+1\right]\nonumber\\
	&\leq
	\frac{\kappa_7}{n}\left[1 +\gamma^n_K + \int_0^T\E\sup_{0\leq u\leq s}\left|\sum_{j=0}^K j\pi^n_j(u)\right|^2ds\right],\label{eqn:MsecMoment}
	\end{align}
	where $\gamma^n_K = \E(K^2 \sup_{0\leq s\leq T} \pi^n_{K+1}(s))$.
	Combining  \eqref{eqn:firstmomExpan}, \eqref{eq:eq1030}, \eqref{eqn:MsecMoment}, and using the fact that $\left|\sum_{j=0}^\iy j\pi^n_j(0)\right|\leq c_{\pi(0)}$,
	\begin{align*}
		\E\sup_{0\leq t\leq T}\left|\sum_{j=0}^K j\pi^n_j(t)\right|^2&\leq
		\kappa_8 \left(1 + \EE\sup_{0\leq t\leq T} \left|\sum_{j=0}^K jA^n_j(t)\right|^2 + \EE\sup_{0\leq t\leq T} \left|\sum_{j=0}^K jM^n_j(t)\right|^2 \right)\\
		&\le \kappa_9\left(1 + \gamma^n_K + \frac{1}{n}\int_0^T\E\sup_{0\leq s\leq t}\left|\sum_{j=0}^K j\pi^n_j(s)\right|^2 ds\right).
	\end{align*}
	By Gronwall's lemma (since the above inequality also holds for all $T_1\le T$), there is a $\kappa_{10}\in (0,\infty)$ such that for all $n,K\in \NN$
	\begin{align*}
		\E\sup_{0\leq t\leq T}\left|\sum_{j=0}^K j\pi^n_j(t)\right|^2&\leq \kappa_{10}\left(1 + \gamma^n_K\right).
	\end{align*}		
	Sending $K\to \infty$ and recalling from Lemma \ref{lem:tailbound} that for each fixed $n$, as $K\to \infty$,
	$\gamma^n_K \to 0$ 
	%and $K \EE\sup_{0\leq t\leq T} (\pi^n_K(t)+ \pi^n_{K+1}(t)) \to 0$,
	 we have for all $n$
	\begin{equation*}
	\E\sup_{0\leq t\leq T}\left|\sum_{j=0}^{\infty} j\pi^n_j(t)\right|^2\leq \kappa_{10}
	\end{equation*}
	where $\kappa_{10}$ is independent of $n$. This proves \eqref{eqn:probmzrtight}. Finally, \eqref{eqn:secmomentbound} follows from \eqref{eqn:probmzrtight} upon sending $K\to\iy$ in \eqref{eqn:piSecMoment2}.
\end{proof}

\subsection{Tightness}\label{sec:diffTightness}
We now proceed with the proof of tightness of $\{(X^n,\bar{M}^n)\}_{n\in\NN}$.
Let for $M\in\RR_+$,
\begin{equation*}
\clv_M \doteq \left\{r\in\cls\bigg|\sum_{i=0}^\iy ir_i\leq M\right\},
\end{equation*}
where $\clv_M$ is equipped with the topology inherited from $\bfell_2$.
We begin by establishing the following Lipschitz property for $F$ on $\clv_M$.
\begin{lemma}\label{lem:lipschitz2}
The map $F$ is a Lipschitz function from $\clv_M$ to $\bfell_2$ for each $M\in\RR_+$. Namely, there exists an $C(M)\in(0,\iy)$ such that for any $r,\ti r\in\clv_M$, 
\begin{equation}\label{eqn:LipschitzDisplay}
\|F(r)-F(\ti r)\|_2\leq C(M)\|r-\ti r\|_2.
\end{equation}
\end{lemma}
\begin{proof}
Fix $M\in\RR_+$. Let $r,\ti r\in\clv_M$ and, for  $i_1\in\NN_0$ and $j,i_2\in\NN$, recall $R_{j,i_1,i_2}(r,\ti r)$ from \eqref{eqn:Rdef}.
Using \eqref{eqn:tripleIneq} and the fact that $r,\ti r\in\cls$, we have
\begin{equation*}
\begin{aligned}
(R_{j,i_1,i_2}(r,\ti r))^2
&\leq 3[r_{j}^{i_2}-\ti r_{j}^{i_2}]^2\\
&\qquad +3\ti r_{j}^{2i_2}\left[\left(\sum_{m=j+1}^\iy r_m\right)^{L-i_1-i_2}-\left(\sum_{m=j+1}^\iy \ti r_m\right)^{L-i_1-i_2}\right]^2\\
&\qquad+3\ti r_{j}^{2i_2}\left[\left(\sum_{m=0}^{j-1} r_m\right)^{i_1}-\left(\sum_{m=0}^{j-1} \ti r_m\right)^{i_1}\right]^2.
\end{aligned}
\end{equation*}
By an argument similar to the one used to derive \eqref{eqn:Rineq} and an application of the Cauchy Schwartz inequality we have the following inequality for all $i_1,i_2\leq L,\ i_2>0,$
\begin{equation}\label{eqn:Rineq2}
\begin{aligned}
 \left(R_{j,i_1,i_2}(r,\ti r)\right)^2
%&\leq \kappa_1\left([r_{j}-\ti r_{j}]^2+(j+1)\ti r^{2i_2}_{j}\|r-\ti r\|_2^2\right)\\
&\leq \kappa_1\left([r_{j}-\ti r_{j}]^2 +(j+1)\ti r_{j}\|r-\ti r\|_2^2\right).
\end{aligned}
\end{equation}
Using arguments analogous to those in the proof of Lemma \ref{lem:lipschitz} we have
\begin{equation}\label{eqn:Fineq2}
\begin{aligned}
\|F(r)-F(\ti r)\|_2
&\leq  \kappa_2\lambda L!\left(\sum_{j=0}^\iy\sum_{i_1=0}^{k-1}\sum_{i_2=1}^{L-i_1}[R_{j,i_1,i_2}(r,\ti r)]^2\right)^{1/2}+ 2k\left(\sum_{j=0}^\iy(r-\ti r)_j^2\right)^{1/2}\\
&\leq \kappa_3\left(\sum_{j=0}^\iy\left[[r_{j}-\ti r_{j}]^2 +(j+1)\ti r_{j}\|r-\ti r\|_2^2\right]\right)^{1/2}+ 2k\|r-\ti r\|_2\\
&\leq \kappa_4\|r-\ti r\|_2\left(1 +\sum_{j=0}^\iy j\ti r_{j}\right)^{1/2}+ 2k\|r-\ti r\|_2.
\end{aligned}
\end{equation}
Since $r,\ti r\in\clv_M$, \eqref{eqn:Fineq2} gives
\begin{equation*}
\|F(r)-F(\ti r)\|_2
\leq \kappa_4(M+1)^{1/2}\|r-\ti r\|_2+ 2k\|r-\ti r\|_2
\end{equation*}
and thus with $C(M)\doteq \kappa_4(M+1)^{1/2}+2k$, \eqref{eqn:LipschitzDisplay} is satisfied for all $r,\ti r\in\clv_M$ which proves the result.
\end{proof}

Recall the process $X^n$ introduced in \eqref{eqn:Xndef} and $\bar{M}^n$ defined below \eqref{eqn:Abardef}.
The following proposition gives tightness of $\{(X^n,\bar{M}^n)\}_{n\in\NN}$.
\begin{proposition}\label{prop:diffTightness}
Suppose that $\{\pi^n\}_{n\in\NN}$ is as in the statement of Theorem \ref{thm:LLN} with $\sup_{n\in\NN}\sum_{j=0}^\iy  j^2\pi^n_j(0)<\iy$. Let $X^n(0)=\sqrt{n}(\pi^n(0)-\pi_0)$ and suppose that \eqref{eqn:XinitBound} is satisfied.
Then $\{(X^n,\bar{M}^n)\}_{n\in\NN}$ is a $\CC$-tight sequence of $\DD([0,T]:(\bfell_2)^2)$-valued random variables.
\end{proposition}
\begin{proof}
We will make use of Theorem \ref{thm:AppSemiMartTight} in the Appendix.
We first prove that $\{\bar{M}^n\}_{n\in\NN}$ is tight. 
In order to show that condition ($A$) in Theorem \ref{thm:AppSemiMartTight} is satisfied for $\{\bar{M}^n\}_{n\in\NN}$ it suffices (cf. Theorem 2.3.2 in \cite{joffe1986weak}) to show that the condition is satisfied for the real-valued process $\lan\bar{M}^n\ran(t)\doteq \sum_{j=0}^\iy\lan\bar{M}^n_j\ran(t)$.
Fix 
$\varepsilon \in (0,T]$ and $N \in [0, T-\varepsilon]$. Let $\tau_n\leq N$ be a sequence of $\{\clf^n_t\}$-stopping times. Then, \eqref{eqn:jexpansion5} and \eqref{eqn:zetaineq} imply that
for $\theta \in [0, \varepsilon]$
\begin{equation*}
\begin{aligned}
&|\lan\bar{M}^n(\tau_n+\theta)\ran-\lan\bar{M}^n(\tau_n)\ran|\\
&\qquad=\left|\sum_{j=0}^\iy\left[\int_{\tau_n}^{\tau_n+\theta}\sum_{\ell\in\Sigma}\lan\Del_\ell, e_j\ran_2 \frac{\lambda}{I(n)}\prod_{i=0}^\iy\binom{n\pi^n_i(s)}{\rho_i(\ell)}+k\int_{\tau_n}^{\tau_n+\theta}[\pi_{j+1}^n(s)-\pi_j^n(s)]ds\right]\right|\\
&\qquad\leq \kappa_{1}\sum_{j=0}^\iy\int_{\tau_n}^{\tau_n+\theta}[\pi_j^n(s)+\pi_{j-1}^n(s) + \pi_{j+1}^n(s)]ds\\
&\qquad\leq \kappa_{1}\varepsilon.
\end{aligned}
\end{equation*}
Proof of $(A)$ is now immediate. 

We next show that $\{\bar{M}^n\}_{n\in\NN}$ satisfies condition ($T_1$) in Theorem \ref{thm:AppSemiMartTight}. 
For this we will apply Theorem \ref{thm:AppHilbertTight}.
We first verify  $\{\bar{M}^n(t)\}_{n\in\NN}$ satisfies ($a$) of Theorem \ref{thm:AppHilbertTight} for all $t\in[0,T]$.
It follows from \eqref{eqn:quadMRate} that
\begin{equation}\label{eqn:24.9}
\sup_{n\in\NN}\E\lan \bar{M}^n\ran(T)=\sup_{n\in\NN}n\E\lan M^n\ran(T) \leq\kappa_2.
\end{equation}
This, combined with Doob's inequality, implies for each $N\in\NN$
\begin{equation*}
\sup_{n\in\NN}\sum_{i=0}^N\E\sup_{0\leq t\leq T}\left|\bar{M}_i^n(t)\right|
\leq N+\sup_{n\in\NN}\sum_{i=0}^N\E (\sup_{0\leq t\leq T}\bar{M}_i^n(t))^2
\leq N+\kappa_3.
\end{equation*}
Using Markov's inequality, part ($a$) of Theorem \ref{thm:AppHilbertTight} follows.

We now verify condition (b) in Theorem \ref{thm:AppHilbertTight} for $\{\bar{M}^n(t)\}_{n\in\NN}$ for each fixed $t\in[0,T]$. Note that $\lan\bar{M}^n_j\ran(t) = n\lan M^n_j\ran(t)$ and thus, from \eqref{eqn:Mjs} and \eqref{eqn:Zineq},
\begin{equation}\label{eqn:tripleStar}
\lan\bar{M}^n_j\ran(t)\leq \kappa_4\int_0^t(\pi^n_{j-1}(s)+\pi^n_{j}(s)+\pi^n_{j+1}(s))ds.
\end{equation}
It follows from  \eqref{eqn:tripleStar} and the Cauchy-Schwartz inequality that
\begin{equation}\label{eqn:MsqSum}
\begin{aligned}
\sum_{j=N}^\iy \E(\bar{M}_j^n(t))^2
&= \sum_{j=N}^\iy \E\lan \bar{M}_j^n(t)\ran
\leq \kappa_5\E\int_0^t\sum_{j=N-1}^\iy\pi^n_j(s)ds\\
&\leq \kappa_5\left(\sum_{j=N-1}^\iy\frac{1}{j^2}\right)^{1/2}\int_0^t\E\left(\sum_{j=N-1}^\iy j^2(\pi^n_j(s))^2\right)^{1/2}ds.
\end{aligned}
\end{equation}
From \eqref{eqn:secmomentbound},
\begin{equation}\label{eqn:piSecMoment3}
\sup_{n\in\NN}\E\int_0^T\sum_{j=0}^\iy j^2(\pi^n_j(s))^2ds
\leq\sup_{n\in\NN}\E\int_0^T\sum_{j=0}^\iy j^2\pi_j^n(s)ds 
\doteq \kappa_6
<\iy.
\end{equation} 
Using this observation in \eqref{eqn:MsqSum} we have
\begin{equation*}
\sum_{j=N}^\iy \E(\bar{M}_j^n(t))^2
\leq \kappa_7\left(\sum_{j=N-1}^\iy\frac{1}{j^2}\right)^{1/2}\int_0^t\E\left(\sum_{j=N-1}^\iy j^2\pi^n_j(s)\right)^{1/2}ds\leq \kappa_8\left(\sum_{j=N-1}^\iy\frac{1}{j^2}\right)^{1/2}.
\end{equation*}
From Markov's inequality we now see that for any $\del>0$
\begin{equation*}
\lim_{N\to\iy}\sup_{n\in\NN}\PP\left(\sum_{j=N}^\iy(\bar{M}^n_j(t))^2>\del\right)
=0
\end{equation*}
which verifies part (b) of Theorem \ref{thm:AppHilbertTight}. Thus we have shown that $\{\bar{M}^n(t)\}_{n\in\NN}$ is a tight sequence of $\bfell_2$-valued random variables for all $t\in[0,T]$. From Theorem \ref{thm:AppSemiMartTight} it now follows that $\{\bar{M}^n\}_{n\in\NN}$ is a tight sequence of $\DD([0,T]:\bfell_2)$-valued random variables.

We will now argue that $\{X^n\}_{n\in\NN}$ is a tight sequence of $\DD([0,T]:\bfell_2)$-valued random variables. Again, via Theorem \ref{thm:AppSemiMartTight}, it suffices to show that $\{X^n(t)\}_{n\in\NN}$ is tight for every $t\in[0,T]$ (which will follow from verifying conditions ($a$) and ($b$) in Theorem \ref{thm:AppHilbertTight}) and that $\{X^n\}_{n\in\NN}$ satisfies condition $(A)$ of Theorem \ref{thm:AppSemiMartTight}.  
We first show that,  for all $t\in[0,T]$, condition ($a$) in Theorem \ref{thm:AppHilbertTight} holds for $\{X^n(t)\}_{n\in\NN}$. Namely we show that for each $N\in\NN$ and $t\in[0,T]$
\begin{equation}\label{eqn:25b}
\lim_{A\to\iy}\sup_{n\in\NN}\PP\left(\sum_{j=0}^N|X^n_j(t)|>A\right)=0.
\end{equation}
Fix $\eps>0$. From Lemma \ref{lem:mombd}, there is a $M\in(0,\iy)$ such that 
\begin{equation}\label{eqn:epsM}
\sup_{n\in\NN}\E\left(\sup_{0\leq t\leq T}\sum_{j=0}^\iy j\pi_j^n(t)\right)\leq \frac{M\eps}{2}.
\end{equation}
Let $B_M^n\doteq\{\sup_{0\leq t\leq T}\sum_{j=0}^\iy j\pi^n_j(t)\leq M\}$. Then for $t\in[0,T]$ and $N\in\NN$
\begin{equation}\label{eqn:probaexpansion}
\begin{aligned}
\PP\left(\sum_{j=0}^N|X^n_j(t)|>A\right)
&\leq \PP\left(\sup_{0\leq t\leq T}\sum_{j=0}^\iy j\pi^n_j(t)>M\right)
	+\PP\left(\sum_{j=0}^N|X^n_j(t)|>A,B_M^n\right)\\
&\leq \frac{\eps}{2}+\PP\left(\sum_{j=0}^N|X^n_j(t)|>A,B_M^n\right).
\end{aligned}
\end{equation}
The Cauchy-Schwartz inequality yields
\begin{equation}\label{eqn:Xexpansion}
\sum_{i=0}^N|X^n_j(t)|
\leq \sqrt{N}\left(\sum_{j=0}^N|X^n_j(t)|^2\right)^{1/2}
\leq \sqrt{N}\|X^n(t)\|_2.
\end{equation}
Furthermore, from \eqref{eqn:semimartrep} and the triangle inequality,
\begin{equation}\label{eqn:54b}
\|X^n(t)\|_2
\leq \|X^n(0)\|_2+\|\bar{A}^n(t)\|_2+\|\bar{M}^n(t)\|_2.
\end{equation}
The definition of $\bar{A}^n$ in \eqref{eqn:Abardef} gives
\begin{equation*}
\|\bar{A}^n(t)\|_2=\sqrt{n}\left\|A^n(t)-\int_0^tF(\pi(s))ds\right\|_2.
\end{equation*}
The moment bound \eqref{eqn:probmzrtight} proved in Lemma \ref{lem:mombd} implies
\begin{equation}\label{eqn:firtMomSquaredBound}
\sup_{n\in\NN}\E\sup_{0\leq t\leq T}\left|\sum_{j=0}^\iy j\pi^n_j(t)\right|^2\doteq\kappa_7<\iy.
\end{equation}
and thus, for some $\kappa_8\in(0,\iy)$,
\begin{equation}\label{eqn:firstMomSquaredBoundLim}
\sup_{0\leq t\leq T}\left|\sum_{j=0}^\iy j\pi_j(t)\right|^2\leq\kappa_8
\end{equation}
as well.
From \eqref{eqn:AtoF} and the Lipschitz property proved in Lemma \ref{lem:lipschitz2}, with $M\geq \kappa_7\vee\kappa_8$ on the set $B^n_M$,
\begin{equation*}
\begin{aligned}
\|\bar{A}^n(t)\|_2
\leq \sqrt{n}\int_0^t\|F(\pi^n(s))-F(\pi(s))\|_2ds + \frac{\kappa_9}{\sqrt{n}}
\leq C(M)\int_0^t\|X^n(s)\|_2ds + \frac{\kappa_9}{\sqrt{n}}.
\end{aligned}
\end{equation*}
Thus, from \eqref{eqn:54b} and Gronwall's lemma, on the set $B_M^n$, for all $n\geq 1$
\begin{equation}\label{eqn:Xngronwall}
\sup_{0\leq t\leq T}\|X^n(t)\|_2\leq \kappa_{10}\left(\frac{1}{\sqrt{n}}+\|X^n(0)\|_2+\sup_{0\leq t\leq T}\|\bar{M}^n(t)\|_2\right)e^{C(M)T}.
\end{equation}
From \eqref{eqn:24.9} and Doob's inequality 
\begin{equation}\label{eqn:Mnunifsecmomentbound}
\sup_{n\in\NN}\E\sup_{0\leq t\leq T}\|\bar{M}^n(t)\|^2_2<\iy.
\end{equation}
Also by assumption, $X^n(0)\to x_0$ in $\bfell_2$. Thus for the given $\eps>0$, we can find $\alpha_0$ such that for all $\alpha\geq \alpha_0$
\begin{equation*}
\PP\left(\sup_{0\leq t\leq T}\|X^n(t)\|_2\geq \frac{\alpha}{\sqrt{N}}, B^n_M\right)\leq\frac{\eps}{2}.
\end{equation*}
Therefore from \eqref{eqn:probaexpansion} and \eqref{eqn:Xexpansion} we have that for all $A\geq \frac{\alpha_0}{\sqrt{N}}$,
\begin{equation*}
\sup_{n\in\NN}\PP\left(\sum_{j=0}^N|X^n_j(t)|>A\right)\leq \frac{\eps}{2}+\frac{\eps}{2}=\eps.
\end{equation*}
Since $\eps>0$ is arbitrary we get \eqref{eqn:25b}. Thus, we have verified part ($a$) of Theorem \ref{thm:AppHilbertTight} for $\{X^n(t)\}_{n\in\NN}$, for each $t\in[0,T]$. 

We next consider part ($b$) of Theorem \ref{thm:AppHilbertTight}.
Namely, we show that for every $\del>0$ and $t\in[0,T]$,
\begin{equation*}
\lim_{N\to\iy}\sup_{n\in\NN}\PP\left(\sum_{j=N}^\iy(X^n_j(t))^2>\del\right)=0.
\end{equation*}
For this it suffices to show that
\begin{equation}\label{eqn:XnsecondMoment}
\sup_{n\in\NN}\E\sup_{0\leq t\leq T}\sum_{j=0}^\iy j^2(X^n_j(t))^2<\iy.
\end{equation}
Recalling that $X^n_j(t) = X^n_j(0)+\bar{A}^n_j(t)+\bar{M}^n_j(t)$ for each $j\in\NN$, it follows that
\begin{equation}\label{eqn:XsecMomentExp}
\begin{aligned}
\E\sup_{0\leq t\leq T}\sum_{j=0}^\iy j^2(X^n_j(t))^2
&\leq 3\sup_{n\in\NN}\sum_{j=0}^\iy j^2(X^n_j(0))^2+3\sup_{n\in\NN}\E\sup_{0\leq t\leq T}\sum_{j=0}^\iy j^2(\bar{A}^n_j(t))^2\\
&\qquad+3\sup_{n\in\NN}\E\sup_{0\leq t\leq T}\sum_{j=0}^\iy j^2(\bar{M}^n_j(t))^2.
\end{aligned}
\end{equation}
Using the definitions of $\bar{A}^n$, $A^n$, and $F$ in \eqref{eqn:Abardef}, \eqref{eqn:AjzetaExp}, and \eqref{eqn:codingF}, respectively, we can write
\begin{equation}\label{eqn:AsecMoment1}
\begin{aligned}
(\bar{A}^n_j(t))^2
&\leq \kappa_{11}\Bigg\{\int_0^tn\left[\frac{(n-L)!}{n!}\zeta(j,n\pi^n(s))-\bar{\zeta}(j,\pi(s))\right]^2ds\\
&\qquad +\int_0^tn\left[\frac{(n-L)!}{n!}\zeta(j-1,n\pi^n(s))-\bar{\zeta}(j-1,\pi(s))\right]^2ds\\
&\qquad+ \int_0^tn\left[\pi^n_j(s)-\pi_j(s)\right]^2ds+\int_0^tn\left[\pi^n_{j+1}(s)-\pi_{j+1}(s)\right]^2ds\Bigg\}.
\end{aligned}
\end{equation}
From \eqref{eqn:zetatobar} and in a similar manner as in \eqref{eqn:zetatoR} we have
\begin{equation*}
\begin{aligned}
n\left[\frac{(n-L)!}{n!}\zeta(j,n\pi^n(s))-\bar{\zeta}(j,\pi(s))\right]^2
&\leq \kappa_{12}\left\{(\pi^n_j(s))^2+n\left[\bar{\zeta}(j,\pi^n(s))-\bar{\zeta}(j,\pi(s))\right]^2\right\}\\
&\leq \kappa_{13}\left\{(\pi^n_j(s))^2+n\sum_{i_1=0}^{k-1}\sum_{i_2=1}^{L-i_1}R_{j,i_1,i_2}(\pi^n(s),\pi(s))^2\right\},
\end{aligned}
\end{equation*}
where $R_{j,i_1,i_2}$ is as in \eqref{eqn:Rdef}.
By \eqref{eqn:Rineq2} and the Cauchy Schwartz inequality we now have, 
\begin{equation*}
\begin{aligned}
nR_{j,i_1,i_2}(\pi^n(s),\pi(s))^2
&\leq \kappa_{14}\left[(X^n_j(s))^2+\pi_j(s)\left(\sum_{m=0}^\iy|X^n_m(s)|\right)^2\right]\\
&\leq \kappa_{14}\left[(X^n_j(s))^2+\pi_j(s)\left(\sum_{m=0}^\iy\frac{1}{m^2}\right)\sum_{m=0}^\iy m^2(X^n_m(s))^2\right].
\end{aligned}
\end{equation*}
Therefore
\begin{equation*}
\begin{aligned}
&n\left[\frac{(n-L)!}{n!}\zeta(j,n\pi^n(s))-\bar{\zeta}(j,\pi(s))\right]^2\\
&\qquad\leq \kappa_{15}\left\{(\pi^n_j(s))^2+(X^n_j(s))^2+\pi_j(s)\sum_{m=0}^\iy m^2(X^n_m(s))^2\right\}.
\end{aligned}
\end{equation*}
Combining this estimate with \eqref{eqn:AsecMoment1} and \eqref{eqn:piSecMoment3} yields
\begin{equation}\label{eqn:AsecMoment2}
\begin{aligned}
\E\sum_{j=0}^\iy j^2(\bar{A}^n_j(t))^2
&\leq \kappa_{16}\E\left\{\int_0^t\sum_{j=1}^\iy j^2\left[(X^n_{j-1}(s))^2+(X^n_j(s))^2+(X^n_{j+1}(s))^2\right.\right.\\
&\qquad\left.\left.+(\pi_j(s)+\pi_{j-1}(s))\sum_{m=0}^\iy m^2(X^n_m(s))^2\right]ds\right\}+\kappa_{16}\\
&\leq \kappa_{17}\E\int_0^t\left(1+\sum_{j=1}^\iy j^2\pi_j(s)\right)\left(\sum_{j=1}^\iy j^2(X^n_j(s))^2\right)ds+\kappa_{17}.
\end{aligned}
\end{equation}
Additionally, it follows from \eqref{eqn:tripleStar} and \eqref{eqn:secmomentbound} that,
\begin{equation*}
\E\sup_{0\leq t\leq T}\sum_{j=0}^\iy j^2\bar{M}_j^n(t)^2
\leq \kappa_{18}'\E\sum_{j=0}^\iy j^2\lan\bar{M}_j^n\ran_T
\leq \kappa_{18}\int_0^T\left[1+\E\sum_{j=0}^\iy j^2\pi_j^n(s)\right]ds
\leq \kappa_{19}.
\end{equation*}
Therefore, from \eqref{eqn:XinitBound}, \eqref{eqn:XsecMomentExp}, and \eqref{eqn:AsecMoment2}, for all $t\in[0,T]$,
\begin{equation*}
\E\sup_{0\leq t\leq T}\sum_{j=0}^\iy j^2(X^n_j(t))^2
\leq \kappa_{20}+ \kappa_{20}\int_0^T\left(1+\sum_{j=1}^\iy j^2\pi_j(t)\right)\E\sup_{0\leq s\leq t}\left(\sum_{j=1}^\iy j^2(X^n_j(s))^2\right)dt.
\end{equation*}
From \eqref{eqn:secmomentbound} and Fatou's lemma, $\int_0^T\sum_{j=1}^\iy j^2\pi_j(s)ds<\iy$ and thus by Gronwall's lemma
\begin{equation*}
\sup_{n\in\NN}\E\sup_{0\leq t\leq T}\sum_{j=0}^\iy j^2(X^n_j(t))^2
\leq \kappa_{19} e^{\kappa_{20}\int_0^T\left(1+\sum_{j=1}^\iy j^2\pi_j(s)\right)}
<\iy.
\end{equation*}
This proves \eqref{eqn:XnsecondMoment}  and verifies part ($b$) of Theorem \ref{thm:AppHilbertTight} for $\{X^n(t)\}_{n\in\NN}$ for each $t\in[0,T]$. Thus $\{X^n(t)\}_{n\in\NN}$ is a tight sequence of $\bfell_2$-valued random variables for every $t\in[0,T]$.

We now show that condition $(A)$ of Theorem \ref{thm:AppSemiMartTight} holds for $\{X^n\}_{n\in\NN}$. Since $X^n(t)=X^n(0)+\bar{A}^n(t)+\bar{M}^n(t)$ and we have shown the condition is satisfied by $\{\bar{M}^n\}_{n\in\NN}$, it suffices to show that the condition holds for $\{\bar{A}^n\}_{n\in\NN}$. Let $N,\eta,\eps,\theta>0,\ N\leq T-\theta,$ and suppose $\{\tau_n\}_{n\in\NN}$ is a family of stopping times such that $\tau_n\leq N$. From the definition of $\bar{A}^n$ (cf. \eqref{eqn:Abardef}) and \eqref{eqn:AtoF} we have that
\begin{equation}\label{eqn:AbarFluc}
\left\|\bar{A}^n(\tau_n+\theta)-\bar{A}^n(\tau_n)\right\|_2
\leq \int_{\tau_n}^{\tau_n+\theta}\sqrt{n}\left\|F(\pi^n(t))-F(\pi(t))\right\|_2dt+\frac{\kappa_{21}}{\sqrt{n}}
\end{equation}
where $\kappa_{21}$ is independent of the choice of $\tau_n$ and $N$. Fix $n_0\in\NN$ such that $\eta -\frac{\kappa_{21}}{\sqrt{n_0}}>0$ and let $\eta' = \eta -\frac{\kappa_{21}}{\sqrt{n_0}}$. 
Recall $\kappa_7,\kappa_8$ introduced in \eqref{eqn:firtMomSquaredBound} and \eqref{eqn:firstMomSquaredBoundLim}, and $B_M^n$ introduced below \eqref{eqn:epsM}.
Select $M\in(0,\iy)$ large enough that $M>\kappa_7\vee\kappa_8$ and \eqref{eqn:epsM} holds. Then for all $n\geq n_0$,
\begin{equation}\label{eqn:Xnfluctuation1}
\begin{aligned}
&\PP \left\{\left\|\int_{\tau_n}^{\tau_n+\theta}\sqrt{n}[F(\pi^n(t))-F(\pi(t))]dt\right\|_2>\eta'\right\}\\
&\qquad\leq \PP \left\{\left\|\int_{\tau_n}^{\tau_n+\theta}\sqrt{n}[F(\pi^n(t))-F(\pi(t))]dt\right\|_2>\eta',B_M^n\right\}+\PP \left\{\sup_{0\leq t\leq T}\sum_{j=0}^\iy j\pi_j^n(t)>M\right\}\\
&\qquad\leq \PP \left\{\left\|\int_{\tau_n}^{\tau_n+\theta}\sqrt{n}[F(\pi^n(t))-F(\pi(t))]dt\right\|_2>\eta',B_M^n\right\}+\frac{\eps}{2}.
\end{aligned}
\end{equation}
It follows from the Lipschitz property of $F$ proved in Lemma \ref{lem:lipschitz2} that
\begin{equation}\label{eqn:Xnfluctuation2}
\PP \left\{\int_{\tau_n}^{\tau_n+\theta}\sqrt{n}\left\|F(\pi^n(t))-F(\pi(t))]\right\|_2dt>\eta',B_M^n\right\}\leq \PP \left\{C(M)\int_{\tau_n}^{\tau_n+\theta}\left\|X^n(t)\right\|_2dt>\eta',B_M^n\right\}.
\end{equation}
Recall from \eqref{eqn:Xngronwall} that for some $\ti C(M)\in(0,\iy)$ on the set $B_M^n$
\begin{equation*}
C(M)\sup_{0\leq t\leq T}\|X^n(t)\|_2\leq \ti C(M)(1+\sup_{0\leq t\leq T}\|\bar{M}^n(t)\|_2).
\end{equation*}
Thus from \eqref{eqn:Xnfluctuation2}, Markov's inequality, and \eqref{eqn:Mnunifsecmomentbound} we have
\begin{equation}\label{eqn:Xnfluc5}
\begin{aligned}
\PP \left\{\int_{\tau_n}^{\tau_n+\theta}\sqrt{n}\left\|F(\pi^n(t))-F(\pi(t))]\right\|_2dt>\eta',B_M^n\right\}
&\leq \PP\{\theta\ti C(M)(1+\sup_{0\leq t\leq T}\|\bar{M}^n(t)\|_2)>\eta'\}\\
&\leq \frac{\theta \ti C(M)(1+\E\sup_{0\leq t\leq T}\|\bar{M}^n(t)\|_2)}{\eta'}\\
&\leq \theta\ti C(M)\kappa_{22}
\end{aligned}
\end{equation}
Combining \eqref{eqn:Xnfluctuation1} and \eqref{eqn:Xnfluc5} gives, whenever $\theta\leq \del$,
\begin{equation*}
\sup_{0\leq\theta\leq\del}\PP \left\{\left\|\int_{\tau_n}^{\tau_n+\theta}\sqrt{n}[F(\pi^n(t))-F(\pi(t))]dt\right\|_2>\eta'\right\}
\leq C(M)\kappa_{22}\del+ \frac{\eps}{2}.
\end{equation*}
Selecting $\del$ small enough that the first term on the RHS is less than $\eps/2$ we have,
\begin{equation}\label{eqn:Xnfluc4}
\sup_{0\leq\theta\leq\del}\PP \left\{\left\|\int_{\tau_n}^{\tau_n+\theta}\sqrt{n}[F(\pi^n(t))-F(\pi(t))]dt\right\|_2>\eta'\right\}
\leq \frac{\eps}{2}+\frac{\eps}{2}
=\eps.
\end{equation}
Therefore, combining \eqref{eqn:AbarFluc} and \eqref{eqn:Xnfluc4}, gives
\begin{equation*}
\sup_{n\geq n_0}\sup_{0\leq\theta\leq\del}\PP\left\{\left\|\bar{A}^n(\tau_n+\theta)-\bar{A}^n(\tau_n)\right\|_2>\eta\right\}
\leq\eps
\end{equation*}
which shows that condition $(A)$ of Theorem \ref{thm:AppSemiMartTight} is satisfied for $\{\bar{A}^n\}_{n\in\NN}$. Therefore, as discussed earlier, $\{X^n\}_{n\in\NN}$ is a tight sequence of $\DD([0,T]:\bfell_2)$-valued random variables and thus $\{(X^n,\bar{M}^n)\}_{n\in\NN}$ is a tight sequence of $\DD([0,t]:(\bfell_2)^2)$-valued random variables.

Finally,  the $\CC$-tightness of $\{(X^n,\bar{M}^n)\}_{n\in\NN}$ is immediate from the estimate
\begin{equation*}
j_T(X^n)
= j_T(\bar{M}^n)
\leq \frac{2+2k}{\sqrt{n}},\qquad n\in\NN
\end{equation*} 
which follows as in the proof of Proposition \ref{prop:tightness}.
\end{proof}

\subsection{Convergence}\label{sec:diffConv}

In this section we give the proofs of Proposition \ref{prop:DiffUniqueness} and Theorem \ref{thm:mainResult}. Since we have shown tightness of $\{(X^n,\bar{M}^n)\}_{n\in\NN}$ in Section \ref{sec:diffTightness}, all that remains in order to complete the proof of Theorem \ref{thm:mainResult} is to characterize the weak limit points of this sequence of processes. This will be argued by showing that the limit point of any weakly convergent subsequence of $\{X^n\}_{n\in\NN}$ will be a solution to the SDE \eqref{eqn:limitSDE} and that uniqueness holds for \eqref{eqn:limitSDE} in an appropriate class, which will also prove Proposition \ref{prop:DiffUniqueness}. We begin by establishing a uniform integrability property for the sequence $\{\bar{M}^n\}_{n\in\NN}$.
\begin{lemma}\label{lem:UIMBar}
Suppose  $\{\pi^n\}_{n\in\NN}$ satisfies conditions in Proposition \ref{prop:diffTightness}. Then the  sequence $\{\sup_{0\leq t\leq T}\sum_{j=0}^\iy|\bar{M}_j^n(t)|^2\}_{n\in\NN}$ is uniformly integrable.
\end{lemma}
\begin{proof}
It follows from the Cauchy-Schwartz and Burkholder-Davis-Gundy inequalities that
\begin{equation}\label{eqn:mbarnUI1}
\sup_{n\in\NN}\E\sup_{0\leq t\leq T}\left(\sum_{j=0}^\iy|\bar{M}_j^n(t)|^2\right)^2
\leq \sup_{n\in\NN}\left(\sum_{m=0}^\iy\frac{1}{m^2}\right)\sum_{j=0}^\iy\E\sup_{0\leq t\leq T}j^2|\bar{M}_j^n(t)|^4
\leq \kappa_1\sup_{n\in\NN}\sum_{j=0}^\iy j^2\E[\bar{M}_j^n](T)^2.
\end{equation}
Recalling the definition of $M^n$ from \eqref{eqn:MartRep}, for each $j$, $\E[\bar{M}_j^n](T)^2$ can be written as
\begin{equation*}
\begin{aligned}
\E[\bar{M}_j^n](T)^2
&=\E\left\{\sum_{\ell\in\Sigma}\frac{1}{n}\lan e_j,\Del_\ell\Del_\ell^Te_j\ran_2 N_\ell\left(\frac{n\lambda}{\binom{n}{L}}\int_0^T\prod_{i=0}^\iy \binom{n\pi^n_i(s)}{\rho_i(\ell)}ds\right)\right.\\
&\qquad\left.+ \frac{1}{n}\left[D_j\left(k\int_0^Tn\pi^n_j(s)ds\right)+D_{j+1}\left(k\int_0^Tn\pi^n_{j+1}(s)ds\right)\right]\right\}^2
\end{aligned}
\end{equation*}
The first term in the expectation on the RHS of the above equation corresponds to the stream of incoming jobs assigned to queues of length $j$. 
From \eqref{eqn:incoming1}, \eqref{eqn:innerEval}, \eqref{eqn:Zineq}, and the independence of $N_\ell,N_{\ell'}$ for $\ell\neq\ell'$ we have
\begin{equation*}
\begin{aligned}
&\E\left(\sum_{\ell\in\Sigma}\frac{1}{n}\lan e_j,\Del_\ell\Del_\ell^Te_j\ran_2 N_\ell\left(\frac{n\lambda}{\binom{n}{L}}\int_0^T\prod_{i=0}^\iy \binom{n\pi^n_i(s)}{\rho_i(\ell)}ds\right)\right)^2\leq \kappa_2\E\int_0^T[\pi^n_j(s)+\pi^n_{j-1}(s)]ds.
\end{aligned}
\end{equation*}
Similarly,
\begin{equation*}
\E\left(\frac{1}{n}\left[D_j\left(k\int_0^Tn\pi^n_j(s)ds\right)+D_{j+1}\left(k\int_0^Tn\pi^n_{j+1}(s)ds\right)\right]\right)^2
\leq \kappa_3\E\int_0^T[\pi^n_j(s)+\pi^n_{j+1}(s)]ds.
\end{equation*}
Combining these estimates and using \eqref{eqn:secmomentbound}
\begin{equation*}
\begin{aligned}
\sup_{n\in\NN}\sum_{j=0}^\iy j^2E[\bar{M}_j^n](T)^2
\leq\kappa_4\sup_{n\in\NN}\E\int_0^T\sum_{j=1}^\iy j^2(\pi_{j-1}^n(s)+\pi_{j}^n(s)+\pi_{j+1}^n(s))
<\iy
\end{aligned}
\end{equation*}
which, in view of \eqref{eqn:mbarnUI1}, gives the desired uniform integrability.
\end{proof}

The following lemma together with \eqref{eqn:XnsecondMoment} shows that any weak limit point $X$ of $\{X^n\}_{n\in\NN}$ satisfies $X(t)\in\ti\bfell_2$ for all $t\in[0,T]$, a.s.
\begin{lemma}\label{lem:tibfell2X}
Let $z^n,z$ be $\DD([0,T]:\bfell_2)$-valued random variables such that 
\begin{equation*}
\sup_{0\leq t\leq T}\|z^n(t)-z(t)\|_2\to0\text{ in probability as }n\to\iy.
\end{equation*}
Suppose that $\sup_{n\in\NN}\E\sup_{0\leq t\leq T}\sum_{j=0}^\iy j^2(z_j^n(t))^2<\iy$. Then $\sup_{0\leq t\leq T}\sum_{j=0}^\iy j^2(z_j(t))^2<\iy$ almost surely and $\sup_{0\leq t\leq T}\left|\sum_{j=0}^\iy z_j^n(t)-\sum_{j=0}^\iy z_j(t)\right|\to0$ in probability.
\end{lemma}
\begin{proof}
Let $\kappa= \sup_{n\in\NN}\E\sup_{0\leq t\leq T}\sum_{j=0}^\iy j^2[z^n_j(t)]^2$. Note that
\begin{equation*}
\sup_{n\in\NN}\E\sup_{0\leq t\leq T}\sum_{j=0}^\iy |z_j^n(t)|
\leq \left(\sum_{j=1}^\iy\frac{1}{j^2}\right)^{1/2}\sqrt{\kappa}
<\iy.
\end{equation*}
Also, by Fatou's lemma $\E\sup_{0\leq t\leq T}\sum_{j=0}^\iy j^2(z_j(t))^2\leq\kappa$ and so we have $\sup_{0\leq t\leq T}\sum_{j=0}^\iy |z_j(t)|<\iy$ almost surely as well. Now
\begin{equation*}
\begin{aligned}
&\E\left[\sup_{0\leq t\leq T}\left|\sum_{j=0}^\iy z^n_j(t)-\sum_{j=0}^\iy z_j(t)\right|\wedge 1\right]\\
&\qquad\leq\E\left[\sup_{0\leq t\leq T}\left|\sum_{j=0}^m z^n_j(t)-\sum_{j=0}^m z_j(t)\right|\wedge 1\right]+ \E\left[\sup_{0\leq t\leq T}\left|\sum_{j=m+1}^\iy z^n_j(t)\right|\wedge 1\right]+\E\left[\sup_{0\leq t\leq T}\left|\sum_{j=m+1}^\iy z_j(t)\right|\wedge 1\right]\\
&\qquad\equiv T_1^m(n)+T_2^m(n)+T_3^m(n).
\end{aligned}
\end{equation*}
Then for $\kappa_1\in(0,\iy)$
\begin{equation*}
(T_2^m(n))^2\leq \left(\sum_{j=m+1}^\iy\frac{1}{j^2}\right)\kappa_1\quad\text{ and }\quad(T_3^m(n))^2\leq \left(\sum_{j=m+1}^\iy\frac{1}{j^2}\right)\kappa_1.
\end{equation*}
The result now follows on first sending $n\to\iy$ and then $m\to\iy$.
\end{proof}

The following result that shows that $\Phi(t)$ is a trace class operator will be useful in characterizing the martingale term in the limiting diffusion.
Note that, from the definition \eqref{eqn:squarematrix}, $\Phi(t)$ is a non-negative operator.
\begin{lemma}\label{lem:traceOp}
For each $t\in[0,T]$, $\Phi(t)$ is a non-negative trace class operator. Denote by $a(t)$ the non-negative square root of $\Phi(t)$. Then $\int_0^T\|a(s)\|_{\hs}^2ds<\iy$.
\end{lemma}
\begin{proof}
We first show that $\Phi(t)$ is a trace class operator. 
Since $\Phi(t)$ is non-negative (and hence self-adjoint) it suffices to show
\begin{equation*}
\sum_{j=0}^\iy\lan e_j,\Phi(s)e_j\ran_2<\iy
\end{equation*}
Using an argument similar to that used in the derivation of \eqref{eqn:Zdef} one can write $\lan e_j,\Phi(s)e_j\ran_2$, as
\begin{equation}\label{eqn:limcrossquad}
\begin{aligned}
\lan e_j,\Phi(s)e_j\ran_2
&= \lambda L!\bar{Z}(j,\pi(s))+k(\pi_j(s)+\pi_{j+1}(s))
\end{aligned}
\end{equation}
where the definition of $\bar{Z}$ is analogous to $Z$, given as,
\begin{equation}\label{eqn:ZjDef}
\begin{aligned}
\bar{Z}(j,\pi(s))
&\doteq\sum_{i_1=0}^{k-2}\frac{\left(\sum_{m=0}^{j-1}\pi_m(s)\right)^{i_1}}{i_1!}\sum_{i_2=0}^{L-i_1}\frac{\pi_{j-1}(s)^{i_2}}{i_2!}\sum_{i_3=0}^{L-i_1-i_2}[i_2\wedge (k-i_1)_+-i_3\wedge (k-i_1-i_2)_+]^2\\
&\quad\times\frac{\pi_j(s)^{i_3}}{i_3!}\frac{\left(\sum_{m=j+1}^{\iy}\pi_{m}(s)\right)^{L-i_1-i_2-i_3}}{(L-i_1-i_2-i_3)!}.
\end{aligned}
\end{equation}
Using arguments as in \eqref{eqn:Zineq} and \eqref{eqn:offDiagBound} it is easy to see that there exists $c_{\bar{Z}}\in(0,\iy)$ such that for all $j\in\NN_0$,
\begin{equation}\label{eqn:zbarDiag}
\bar{Z}(j,\pi(s))
\leq c_{\bar{Z}} (\pi_{j-1}(s)+\pi_j(s)).
\end{equation}
From \eqref{eqn:limcrossquad} and \eqref{eqn:zbarDiag} it follows that there exists a $\kappa_1\in(0,M)$ such that,
\begin{equation*}
\sum_{j=0}^\iy\lan e_j,\Phi(t) e_j\ran_2
\leq \kappa_2\sum_{j=0}^\iy[\pi_{j-1}(t)+\pi_{j}(t)+\pi_{j+1}(t)]\leq 3\kappa_1.
\end{equation*}
Therefore, $\Phi(t)$ is a trace class operator. 
Finally, note that
\begin{equation*}
\begin{aligned}
\int_0^T\|a(s)\|_{\hs}^2ds
&= \int_0^T\sum_{j=0}^\iy\lan a(s)e_j,a(s)e_j\ran_2 ds
= \int_0^T\sum_{j=0}^\iy\lan e_j,\Phi(s)e_j\ran_2 ds\leq 3\kappa_1T
\end{aligned}
\end{equation*}
which completes the proof.
\end{proof}

We now proceed with the proofs of Proposition \ref{prop:DiffUniqueness} and Theorem \ref{thm:mainResult}.
\begin{proof}[Proof of Proposition \ref{prop:DiffUniqueness}]
The existence of a $(X(t))_{0\leq t\leq T}$ as in the statement of Proposition \ref{prop:DiffUniqueness} will be proved as part of Theorem \ref{thm:mainResult}. We now consider the second statement in Proposition \ref{prop:DiffUniqueness} and let $(X(t))_{0\leq t\leq T}$, $(\ti X(t))_{0\leq t\leq T}$ be two $\{\clf_t\}$-adapted processes solving \eqref{eqn:limitSDEb} with sample paths in $\CC([0,T]:\bfell_2)$ such that $X(t)\in\ti\bfell_2$ and $\ti X(t)\in\ti\bfell_2$ for all $t$, almost surely. In order to show that $X(t)=\ti X(t)$ for all $t\in[0,T]$ almost surely it suffices to show the following Lipschitz property on $G$: There exists a $C\in(0,\iy)$ such that for all $x,\ti x\in\ti\bfell_2$,
\begin{equation}\label{eqn:Glip}
\sup_{0\leq t\leq T}\|G(x,\pi(t))-G(\ti x,\pi(t))\|_2\leq C\|x-\ti x\|_2.
\end{equation}
Note from \eqref{eqn:codingF}, \eqref{eqn:zetabardef}, and \eqref{eqn:opDFdef} that for $j\in\NN_0$ and $(x,r)\in\ti\bfell_2\times\cls$,
\begin{equation}\label{eqn:Gtoxi}
G_j(x,r) 
= \lambda L! [\xi^1_{j-1}(x,r)-\xi^1_{j}(x,r)+\xi^2_{j-1}(x,r)-\xi^2_{j}(x,r)+\xi^3_{j-1}(x,r)-\xi^3_{j}(x,r)]+k\xi^4_{j}(x)
\end{equation}
where
\begin{equation*}
\begin{aligned}
&\xi^1_{j}(x,r)\doteq \sum_{i_1=0}^{k-1}i_1\frac{\left(\sum_{m=0}^{j-1}r_m\right)^{i_1-1}}{i_1!}\sum_{i_2=1}^{L-i_1}[i_2\wedge(k-i_1)]\frac{(r_j)^{i_2}}{i_2!}\frac{\left(\sum_{m=j+1}^{\iy}r_{m}\right)^{L-i_1-i_2}}{(L-i_1-i_2)!}\sum_{m=0}^{j-1}x_m,\\
&\xi^2_{j}(x,r)\doteq \sum_{i_1=0}^{k-1}\frac{\left(\sum_{m=0}^{j-1}r_m\right)^{i_1}}{i_1!}\sum_{i_2=1}^{L-i_1}i_2[i_2\wedge(k-i_1)]\frac{(r_j)^{i_2-1}}{i_2!}\frac{\left(\sum_{m=j+1}^{\iy}r_{m}\right)^{L-i_1-i_2}}{(L-i_1-i_2)!}x_j,\\
&\xi^3_{j}(x,r)\doteq\sum_{i_1=0}^{k-1}\frac{\left(\sum_{m=0}^{j-1}r_m\right)^{i_1}}{i_1!}\sum_{i_2=1}^{L-i_1}(L-i_1-i_2)[i_2\wedge(k-i_1)]\frac{(r_j)^{i_2}}{i_2!}\frac{\left(\sum_{m=j+1}^{\iy}r_{m}\right)^{L-i_1-i_2-1}}{(L-i_1-i_2)!}\sum_{m=j+1}^{\iy}x_m,
\end{aligned}
\end{equation*} 
and
\begin{equation*}
\xi^4_{j}(x)
= [x_{j+1}-x_{j}].
\end{equation*}
Also, let $\xi^i\doteq(\xi^i_{j})_{j=0}^\iy$ for $i=1,2,3,4$. 
Using the triangle inequality, it suffices to show that \eqref{eqn:Glip} holds with $G$ replaced with $\xi^i, i=1,2,3,4$. 
Since $\pi(t)\in\cls$ for all $t\in[0,T]$
\begin{equation}\label{eqn:xi1Bound}
\begin{aligned}
\sup_{0\leq t\leq T}\|\xi^1(x,\pi(t))-\xi^1(\ti x,\pi(t))\|^2_2
&\leq \kappa_1'\sup_{0\leq t\leq T}\sum_{j=0}^\iy\pi_j(t)^2\left[\sum_{m=0}^{j-1}x_m-\sum_{m=0}^{j-1}\ti x_m\right]^2\\
&\leq \kappa_1'\sup_{0\leq t\leq T}\sum_{j=0}^\iy j\pi_j(t)\|x-\ti x\|_2^2\\
%&\leq \kappa_1'\|x-\ti x\|_2^2\sup_{0\leq t\leq T}\sum_{j=0}^\iy j\pi_j(t)\\
&\leq \kappa_1\|x-\ti x\|_2^2,
\end{aligned}
\end{equation}
where the last inequality is from \eqref{eqn:firstMomSquaredBoundLim}.
Also,
\begin{equation*}
\begin{aligned}
\sup_{0\leq t\leq T}\|\xi^2(x,\pi(t))-\xi^2(\ti x,\pi(t))\|^2_2
\leq \kappa_2\sum_{j=0}^\iy[x_j-\ti x_j]^2
= \kappa_2\|x-\ti x\|_2^2.
\end{aligned}
\end{equation*}
Using the fact that $\sum_{m=0}^\iy x_m=\sum_{m=0}^\iy \ti x_m=0$ and the calculation in \eqref{eqn:xi1Bound}
\begin{equation*}
\begin{aligned}
\sup_{0\leq t\leq T}\|\xi^3(x,\pi(t))-\xi^3(\ti x,\pi(t))\|^2_2
&\leq \kappa_3'\sup_{0\leq t\leq T}\sum_{j=0}^\iy\pi_j(t)^2\left[\sum_{m=j+1}^{\iy}x_m-\sum_{m=j+1}^{\iy}\ti x_m\right]^2\\
&= \kappa_3'\sup_{0\leq t\leq T}\sum_{j=0}^\iy\pi_j(t)^2\left[\sum_{m=0}^{j}\ti x_m-\sum_{m=0}^{j}x_m\right]^2\\
&\leq \kappa_3\|x-\ti x\|_2^2.
\end{aligned}
\end{equation*}
Finally,
\begin{equation*}
\begin{aligned}
\|\xi^4(x)-\xi^4(\ti x)\|^2_2
\leq  \sum_{j=0}^\iy[x_j-\ti x_j]^2+\sum_{j=0}^\iy[x_{j+1}-\ti x_{j+1}]^2
\leq 2\|x-\ti x\|_2^2.
\end{aligned}
\end{equation*}
Combining the above Lipschitz estimates for $\xi^i,\ i =1,2,3,4$, we have \eqref{eqn:Glip} and the result follows. 
\end{proof}

We now proceed to the proof of Theorem \ref{thm:mainResult}.
\begin{proof}[Proof of Theorem \ref{thm:mainResult}]
From Proposition \ref{prop:diffTightness} $\{(X^n,\bar{M}^n)\}_{n\in\NN}$ is $\CC$-tight in $\DD([0,T]:(\bfell_2)^2)$.
Suppose $(X,\bar{M})$ is a weak limit of a subsequence of $\{(X^n,\bar{M}^n)\}_{n\in\NN}$ (also indexed by $\{n\}$) given on some probability space ($\Om,\clf,\PP$). Let $m\in\NN$ and $\clh:(\bfell_2\times\bfell_2)^m\to\RR$ be a bounded and continuous function. For $s\leq t\leq T$ and $0\leq t_1\leq\ldots\leq t_m\leq s$ we let $\xi_i^n=(X^n(t_i),\bar{M}^n(t_i))$ and $\xi_i=(X(t_i),\bar{M}(t_i))$. Then, for all $j\in\NN_0$,
\begin{equation*}
\begin{aligned}
\E\clh(\xi_1,\ldots,\xi_m)[\bar{M}_j(t)-\bar{M}_j(s)] = \lim_{n\to\iy}\E\clh(\xi_1^n,\ldots,\xi_m^n)[\bar{M}^n_j(t)-\bar{M}_j^n(s)]=0
\end{aligned}
\end{equation*}
where the first equality comes from the uniform integrability property proved in Lemma \ref{lem:UIMBar} and the second comes from the fact that $\bar{M}^n$ is a martingale for each $n\in\NN$. It follows that $\bar{M}$ is a $\{\clf_t\}$-martingale where $\clf_t=\sigma\{X(s),\bar{M}(s),s\leq t\}$.

As was shown in \eqref{eqn:Zdef2},
\begin{equation}\label{eqn:crossquad2}
\begin{aligned}
\lan\bar{M}_i^n,\bar{M}_j^n\ran(t)
&= n\lan M_i^n,M_j^n\ran(t)\\
&= \frac{\lambda}{\binom{n}{L}}\int_0^tZ(i,j,n\pi^n(s))ds- k\int_0^t1_{\{i = j+1\}}\pi_i^n(s)\\
&\qquad - k\int_0^t1_{\{i+1 = j\}}\pi_j^n(s)ds+k\int_0^t1_{\{i=j\}}(\pi_j^n(s)+\pi_{j+1}^n(s))ds
\end{aligned}
\end{equation}
(see \eqref{eqn:Zdef} and \eqref{eqn:Zdef4} for definition of $Z$). 
Using similar arguments as in \eqref{eqn:Zdef4}, we have the estimate
\begin{equation*}
\lan e_i,\Phi(s)e_j\ran_2
= \lambda L!\bar{Z}(i,j,\pi(s))-k1_{\{i+1=j\}}\pi_j(s)-k1_{\{i=j+1\}}\pi_i(s)+k1_{\{i=j\}}(\pi_j(s)+\pi_{j+1}(s)),
\end{equation*}
where for $i<j$,
\begin{equation*}
\begin{aligned}
\bar{Z}(i,j,\pi(s))
&\doteq\sum_{i_1=0}^{k-2}\frac{\left(\sum_{m=0}^{i-2}\pi_m(s)\right)^{i_1}}{i_1!}\sum_{i_2=0}^{k-i_1-1}\frac{\pi_{i-1}(s)^{i_2}}{i_2!}\sum_{i_3=0}^{k-i_1-i_2-1}[i_2-i_3]\frac{\pi_{i}(s)^{i_3}}{i_3!}\\
&\quad\times\sum_{i_4=0}^{k-i_1-i_2-i_3-1}\frac{\left(\sum_{m=i+1}^{j-2}\pi_m(s)\right)^{i_4}}{i_4!}\sum_{i_5=0}^{L-\sum_{n=1}^4i_n}\frac{\pi_{j-1}(s)^{i_5}1_{\{j> i+1\}}}{i_5!}\\
&\quad\times\sum_{i_6=0}^{L-\sum_{n=1}^5i_n}\left[(1_{\{j= i+1\}}(i_3-i_5)+i_5)\wedge\left(k-\sum_{n=1}^4i_n\right)_+-i_6\wedge\left(k-\sum_{n=1}^5i_n\right)_+\right]\\
&\quad\times\frac{\pi_{j}(s)^{i_6}}{i_6!}\frac{\left(\sum_{m=j+1}^\iy \pi_m(s)\right)^{L-\sum_{n=1}^6i_n}}{
\left(L-\sum_{n=1}^6i_n\right)!},
\end{aligned}
\end{equation*}
for $i>j$, $\bar{Z}(i,j,\pi(s))\doteq\bar{Z}(j,i,\pi(s))$, and for $i=j$, $\bar{Z}(j,j,\pi(s)) \doteq \bar{Z}(j,\pi(s))$, where $\bar{Z}(j,r)$ is defined in \eqref{eqn:ZjDef}.
Using arguments similar to those used in \eqref{eqn:binomToExp} and \eqref{eqn:zetatobar} one can write
\begin{equation*}
\begin{aligned}
&\left|Z(i,j,n\pi^n(s)) - \frac{n!}{(n-L)!}\bar{Z}(i,j,\pi^n(s))\right|\leq \kappa_1n^{L-1}.
%(\pi^n_{i}(s)\pi^n_{j}(s)+\pi^n_{i-1}(s)\pi^n_{j}(s)+\pi^n_{i}(s)\pi^n_{j-1}(s)+\pi^n_{i-1}(s)\pi^n_{j-1}(s)+\pi_i).
\end{aligned}
\end{equation*}
It follows from this, \eqref{eqn:crossquad2}, \eqref{eqn:limcrossquad}, and the fact that $\pi^n\to\pi$ in probability that
\begin{equation*}
\sup_{0\leq t\leq T}\left|\lan\bar{M}_i^n(t),\bar{M}_j^n(t)\ran- \int_0^t\lan e_i,\Phi(s)e_j\ran_2ds\right|\to 0
\end{equation*}
in probability. A similar argument as in Lemma \ref{lem:UIMBar} shows that $\{\lan\bar{M}^n_i,\bar{M}^n_j\ran_t\}_{n\in\NN}$ is uniformly integrable for each $t\in[0,T]$ and $i,j\in\NN_0$. Applying the above convergence and uniform integrability properties,
\begin{equation*}
\begin{aligned}
&\E\clh(\xi_1,\ldots,\xi_m)[\lan\bar{M}_i,\bar{M}_j\ran_t-\lan\bar{M}_i,\bar{M}_j\ran_s-\int_s^t\lan e_i,\Phi(u)e_j\ran_2du]\\
&\qquad= \lim_{n\to\iy}\E\clh(\xi_1^n,\ldots,\xi_m^n)[\lan\bar{M}_i^n,\bar{M}_j^n\ran_t-\lan\bar{M}^n_i,\bar{M}^n_j\ran_s-\int_s^t\lan e_i,\Phi(u)e_j\ran_2du]=0.
\end{aligned}
\end{equation*}
Also from Lemma \ref{lem:UIMBar} and Fatou's lemma
$\E\sup_{0\leq t\leq T}\sum_{j=0}^\iy|\bar{M}_j(t)|^2<\iy.$
Thus we have that $\bar{M}\doteq (\bar{M}_j)_{j\in\NN_0}$ is a collection of square integrable $\{\clf_t\}$-martingales with
\begin{equation*}
\lan \bar{M}_i,\bar{M}_j\ran (t)
= \int_0^t\lan e_i,\Phi(s)e_j\ran_2ds,\qquad t\in[0,T].
\end{equation*}
From Theorem 8.2 of \cite{da2014stochastic} it now follows that there is a $\bfell_2$-cylindrical Brownian motion $\{(W_t(h))_{0\leq t\leq T}: h\in\bfell_2\}$ on some extension $(\bar{\Om},\bar{\clf},\bar{\PP},\{\bar{\clf}_t\})$ of the filtered probability space $(\Om,\clf,\PP,\{\clf_t\})$ such that
\begin{equation}\label{eqn:limitMartingale}
\bar{M}(t)=\int_0^ta(s)dW(s).
\end{equation} 

Recall the representation of $X^n$ in terms of $\bar{A}^n$ and $\bar{M}^n$ from \eqref{eqn:semiMart}.
We now argue that together with $X^n$ and $\bar{M}^n$, $\bar{A}^n(\cdot)$ converges to $\int_0^\cdot G(X(s),\pi(s))ds$ in $\DD([0,T]:\bfell_2)$, in distribution, as $n\to\iy$ (along the chosen subsequence).  The definition of $\bar{A}^n$ in \eqref{eqn:Abardef} and the estimate in \eqref{eqn:AtoF} imply that
\begin{equation}\label{eqn:Aexp}
\sup_{0\leq t\leq T}\left\|\bar{A}^n(t)
- \int_0^t\sqrt{n}[F(\pi^n(s))-F(\pi(s))]ds\right\|_2\leq\frac{\kappa_2}{\sqrt{n}}.
\end{equation}
For $r,\ti r\in\cls$ such that $(r-\ti r)\in\ti\bfell_2$, the $i$-th component of $F(r)-F(\ti r)$ can be written
\begin{equation*}
\begin{aligned}
 {[}F(r)-F(\ti r){]}_i
&= \int_0^1\frac{\partial}{\partial u}F_i(r u + (1-u)\ti r)du\\
&= \int_0^1G_i((r-\ti r), ru+(1-u)\ti r)du\\
& = G_i(r - \ti r, \ti r)+\int_0^1[G_i((r-\ti r, ru+(1-u)\ti r)-G_i(r-\ti r,\ti r)]du.
\end{aligned}
\end{equation*}
Therefore, observing that $cG_i(x,r)=G_i(cx,r)$ for $c\in\RR$ and $(x,r)\in\ti\bfell_2\times\cls$ and noting from \eqref{eqn:XnsecondMoment} that $X^n(s)\in\ti\bfell_2$ for every $s\in[0,T]$ almost surely, we can write
\begin{equation}\label{eqn:taylor}
\sqrt{n}[F(\pi^n(s))-F(\pi(s))]_i
= G_i(X^n(s),\pi(s))+R_i^n(s)
\end{equation}
where
\begin{equation*}
\begin{aligned}
R^n_i(s) 
&= \int_0^1[G_i(X^n(s),\pi^n(s)u+(1-u)\pi(s))-G_i(X^n(s),\pi(s))]du.
\end{aligned}
\end{equation*} 
Thus
\begin{equation*}
\sqrt{n}[F(\pi^n(s))-F(\pi(s))]
= G(X^n(s),\pi(s))+R^n(s)
\end{equation*}
where $R^n(s)\doteq (R^n_i(s))_{i\in\NN_0}$.
We now show that $\int_0^T\|R^n(s)\|_2ds\to0$ in probability as $n\to\iy$. Since $\sum_{m=0}^jX^n_m(s)= \sum_{m=j+1}^\iy X^n_m(s)$, it follows from \eqref{eqn:tripleIneq} that for $r,\ti r\in\cls$
\begin{equation*}
\begin{aligned}
&\|\xi^i(X^n(s),r)-\xi^i(X^n(s),\ti r)\|_2^2\\
&\qquad\leq \kappa_3'\sum_{j=0}^\iy\left(\sum_{m=0}^{j}|X^n_m(s)|\right)^2\left[[r-\ti r]_j^2+\ti r_j\left(\sum_{i=0}^{j-1}[r-\ti r]_i\right)^2+\ti r_j\left(\sum_{i=j+1}^{\iy}[r-\ti r]_i\right)^2\right]\\
&\qquad\leq \kappa_3\left(\sum_{j=0}^\iy j^2|X^n_j(s)|^2\right)\sum_{j=0}^\iy\left[j\ti r_j\|r-\ti r\|_2^2+[r-\ti r]_j^2\right]
\end{aligned}
\end{equation*}
for $i=1,2,3$. The triangle inequality, \eqref{eqn:Gtoxi}, and the observation that $\sup_{0\leq s\leq T}\sum_{j=0}^\iy j\pi_j(s)<\iy$ (see \eqref{eqn:firstMomSquaredBoundLim}) then implies that
\begin{equation*}
\|G(X^n(s),\pi^n(s)u+(1-u)\pi(s))-G(X^n(s),\pi(s))\|_2^2\leq \kappa_3\left(\sum_{j=0}^\iy j^2|X^n_j(s)|^2\right)\|\pi^n(s)-\pi(s)\|_2^2.
\end{equation*}
Since $\sup_{0\leq s\leq T}\|\pi^n(s)-\pi(s)\|_2\to0$ in probability and, from \eqref{eqn:XnsecondMoment}, $\sup_{n\in\NN}\E\sup_{0\leq s\leq T}\sum_{j=0}^\iy j^2|X^n_j(s)|^2<\iy$, it follows that
\begin{equation*}
\sup_{0\leq u\leq 1}\sup_{0\leq s\leq T}\|G(X^n(s),\pi^n(s)u+(1-u)\pi(s))-G(X^n(s),\pi(s))\|_2\to 0
\end{equation*}
in probability, as $n\to\iy$
and thus
\begin{equation}\label{eq:Eq504}
\int_0^T\|R_n(s)\|_2ds
\to 0, \text{ in probability.}
\end{equation}
In view of \eqref{eqn:Aexp}, \eqref{eqn:taylor}, and \eqref{eq:Eq504} it now suffices to show that, along the subsequence
$$\left(X^n, \bar{M}^n, \int_0^\cdot G(X^n(s),\pi(s))ds\right) \Rightarrow \left(X, \bar M, \int_0^\cdot G(X(s),\pi(s))ds\right)$$
in $\DD([0,T]:(\bfell_2)^3)$.
By appealing to the Skorohod representation theorem we can assume without loss of generality that $(X^n,\bar{M}^n)$ converges almost surely in $\DD([0,T]:(\bfell_2)^2)$ to $(X,\bar M)$.
From \eqref{eqn:XnsecondMoment} and Fatou's lemma we also have
\begin{equation*}
\sup_{0\leq t\leq T}\sum_{j=0}^\iy j^2(X_j(t))^2<\iy\quad\text{a.s.}
\end{equation*}
Also, since $\sum_{j=0}^\iy X^n_j(t)=0$ for all $t\in[0,T]$ and $n\in\NN$, by Lemma \ref{lem:tibfell2X} and \eqref{eqn:XnsecondMoment}, we have that $\sum_{j=0}^\iy X_j(t)=0$ for all $t\in[0,T]$ almost surely as well. It then follows that $X^n(t),X(t)\in\ti\bfell_2$ for all $t\in[0,T]$ almost surely for all $n\in\NN$. From the Lipschitz property in \eqref{eqn:Glip} it now follows that, as $n\to\iy$,
\begin{equation}\label{eqn:Gconv}
\int_0^T\|G(X^n(s),\pi(s))-G(X(s),\pi(s))\|_2ds
\leq C\int_0^T\|X^n(s)-X(s)\|_2ds
\to 0,
\end{equation} 
which proves the desired convergence.
% Combining this with  \eqref{eqn:taylor} and \eqref{eqn:Aexp} we have that 
% $$\sup_{0\leq t\leq T}\left\|\bar{A}^n(t)-\int_0^tG(X(s),\pi(s))ds\right\|_2\to0$$
% in probability as $n\to\iy$. 
Together with \eqref{eqn:semiMart} and the representation \eqref{eqn:limitMartingale} we now have that the limit point $(X,\bar{M})$ satisfies
\begin{equation*}
X(t) = x_0 + \int_0^tG(X(s),\pi(s))ds + \int_0^ta(s)dW(s)
\end{equation*}
almost surely for all $t\in[0,T]$. Since $X(t)\in\ti\bfell_2$ for all $t\in[0,T]$ almost surely, this in particular proves the existence part of Proposition \ref{prop:DiffUniqueness}. Finally the uniqueness part of Proposition \ref{prop:DiffUniqueness} (which was established earlier in this section) now says that $X^n$ converges in distribution along the full sequence to the unique weak solution of \eqref{eqn:limitSDE} with values in $\ti\bfell_2$. The result follows.
\end{proof}

\section{Numerical Results}\label{sec:example}
In this section, we present some simulation results comparing the pre-limit $n$-server system with results of the corresponding law of large number and central limit approximations. We consider a system with $n=10,000$ servers. 
For all combinations of $L$ and $k$ in the set $\{(L,k)\in\NN\times\NN:2\leq L\leq 5,k< L\}$, we simulate 1,000 realizations of both the $n$-server system and the diffusion approximation given through Theorem \ref{thm:mainResult} using parameters $T=10$, $\lambda = .9$, and $c=1$. Note that since the limiting processes are infinite dimensional we must truncate to a finite dimensional approximation in order to perform simulations. In our numerical approximations, we truncate to the first 20 coordinates. All computations were performed in Matlab. A numerical ODE solver (ode45) was used to compute the ODE corresponding to the law of large number limit. The limit diffusion was simulated using Euler's method with step sizes of .1. The realizations of the diffusion were used to create 95\% confidence intervals for the following metrics at time $T$; the number of empty queues, the number of ``large'' queues (queues with more than 5 jobs), and the mean queue length. The coverage rates (i.e. the proportion of the $n$-server system simulations which fall within the 95\% confidence interval estimated by the diffusion approximation) can be found in Tables \ref{table:empty}, \ref{tab:Large}, and \ref{tab:mean}. 
\begin{table}[h]
\begin{minipage}{.44\textwidth}
\centering
\begin{tabular}{|c|c|c|c|c|c|}
\cline{3-6}
\multicolumn{2}{c|}{\multirow{2}{*}{}} &\multicolumn{4}{|c|}{L}\\
\cline{3-6}
 \multicolumn{2}{c|}{} & 2 & 3 & 4 & 5\\
\hline
\multirow{4}{*}{k} &1 & 95.1\% & 96.3\% & 97.7\% & 95.9\%\\
\cline{2-6}
&2 & - & 96.5\% & 95.3\% & 95.6\% \\
\cline{2-6}
&3 & - & - & 96.8\% & 97.5\% \\
\cline{2-6}
&4 & -& - & - & 97.1\% \\
\hline
\end{tabular}
\caption{Empty Queue Coverage Rate}\label{table:empty}
\end{minipage}
\begin{minipage}{.44\textwidth}
\centering
\begin{tabular}{|c|c|c|c|c|c|}
\cline{3-6}
\multicolumn{2}{c|}{\multirow{2}{*}{}} &\multicolumn{4}{|c|}{L}\\
\cline{3-6}
 \multicolumn{2}{c|}{}& 2 & 3 & 4 & 5\\
\hline
\multirow{4}{*}{k} &1 & 97.1\% & 100\% & 100\% & 100\%\\
\cline{2-6}
&2 &  & 94.9\% & 95.6\%& 100\% \\
\cline{2-6}
&3 & - & - & 96.7\% & 96.4\% \\
\cline{2-6}
&4 & - & - & - & 95.0\%\\
\hline
\end{tabular}
\caption{Large Queue Coverage Rate}\label{tab:Large}
\end{minipage}\\
\begin{minipage}{.44\textwidth}
\centering
\begin{tabular}{|c|c|c|c|c|c|}
\cline{3-6}
\multicolumn{2}{c|}{\multirow{2}{*}{}} &\multicolumn{4}{|c|}{L}\\
\cline{3-6}
 \multicolumn{2}{c|}{}& 2 & 3 & 4 & 5\\
\hline
\multirow{4}{*}{k} &1 & 95.2\% & 94.8\% & 94.8\% & 95.4\%\\
\cline{2-6}
&2 &  & 94.7 & 92.9\% & 94.9\%\\
\cline{2-6}
&3 & - & - & 96.8\% & 95.1\%\\
\cline{2-6}
&4 & - & - & - & 94.8\%\\
\hline
\end{tabular}
\caption{Mean Queue Length Coverage Rate}\label{tab:mean}
\end{minipage}
\end{table}
As is seen in these results, the diffusion approximation based confidence intervals, in general, contain approximately 95\% of the $n$-server simulated observations, as desired.

The goal of this paper was to develop reliable approximations of the $n$-server system that are much quicker to simulate. Table \ref{table:times} presents the average time (in seconds) required to simulate one trial of the finite system (left) and diffusion approximation (right). As is seen from these tables, the time required to simulate the diffusion approximations is substantially smaller than for the underlying $n$-server jump-Markov process.
\begin{table}[h]
\begin{minipage}{.44\textwidth}
\centering
\begin{tabular}{|c|c|c|c|c|c|}
\cline{3-6}
\multicolumn{2}{c|}{\multirow{2}{*}{}} &\multicolumn{4}{|c|}{L}\\
\cline{3-6}
 \multicolumn{2}{c|}{}& 2 & 3 & 4 & 5\\
\hline
\multirow{4}{*}{k} &1 & 22.6 & 23.8 & 25.4 & 19.2\\
\cline{2-6}
& 2 & - & 39.1 & 38.0 & 33.1 \\
\cline{2-6}
& 3 & - & - & 44.5 & 45.5\\
\cline{2-6}
& 4 & - & - & - & 57.4 \\
\hline
\end{tabular}
\subcaption{Average Time for Finite System}
\end{minipage}
\begin{minipage}{.44\textwidth}
\centering
\begin{tabular}{|c|c|c|c|c|c|}
\cline{3-6}
\multicolumn{2}{c|}{\multirow{2}{*}{}} &\multicolumn{4}{|c|}{L}\\
\cline{3-6}
 \multicolumn{2}{c|}{}& 2 & 3 & 4 & 5\\
\hline
\multirow{4}{*}{k} &1 & .29 & .50 & .79 & .79\\
\cline{2-6} 
& 2 & - & 2.4 & 3.7 & 4.6 \\
\cline{2-6}
&3 & - & - & 6.0 & 10.0 \\
\cline{2-6}
&4 & -  & - & - & 16.3 \\
\hline
\end{tabular}
\subcaption{Average Time for Limit Diffusion}
\end{minipage}
\caption{Average Simulation Times}\label{table:times}
\end{table}
In addition, increasing $n$ will further increase the amount of time required to simulate the $n$-server system. Indeed, $n=10,000$ is a small number compared to the size of typical data centers and server farms that have machines which number in the hundreds of thousands.

\setcounter{equation}{0}
\appendix

\numberwithin{equation}{section}

\section{Auxiliary Results}
\subsection{Criterion for Tightness of Hilbert-Valued Random Variables}
The following theorem gives sufficient conditions for tightness of a sequence of random variables taking values in a (possibly infinite-dimensional) Hilbert space. For a proof see Corollary 2.3.1 of \cite{KallianpurXiong}.
\begin{theorem}\label{thm:AppHilbertTight}
Let $\HH$ be a separable Hilbert Space with inner product $\lan\cdot,\cdot\ran$ and complete orthonormal system $\{e_i\}_{i=1}^\iy$. Suppose $\{Y_n\}_{n\in\NN}$ is a sequence of $\HH$-valued random variables satisfying the following conditions:
\begin{enumerate}
\item[a)] For each $N\in\NN$, $\lim_{A\to\iy}\sup_{n\in\NN}\PP\left(\max_{1\leq i\leq N}\lan Y_n, e_i\ran^2>A\right)=0$
\item[b)] For every $\del>0$, $\lim_{N\to\iy}\sup_{n\in\NN}\PP\left(\sum_{j=N}^\iy \lan Y_n,e_j\ran^2>\del\right)=0$.
\end{enumerate}
Then $\{Y_n\}_{n\in\NN}$ is a tight sequence of $\HH$-valued random variables.
\end{theorem}

\subsection{Criterion for Tightness of RCLL Processes}
The following theorem gives a criterion for tightness of a sequence of RCLL processes with values in a Polish space, see \cite{kurtz1981approximation}.
\begin{theorem}\label{thm:AppSemiMartTight}
Let $\SSS$ be a Polish Space and $\{Y_n\}_{n\in\NN}$ be a sequence of $\DD([0,T]:\SSS)$-valued $\{\clf^n_t\}$-semimartingales satisfying the following conditions:
\begin{enumerate}
\item[$(T_1)$] $\{Y_n(t)\}_{n\in\NN}$ is tight for every $t$ in a dense subset of $[0,T]$.
\item[$(A)$ ] For each  $\eps>0,\ \eta>0$ and $N \in [0,T-\eps]$ there exists a $\del>0$ and $n_0$ with the property that for every collection of stopping times $(\tau_n)_{n\in\NN}$ ($\tau_n$ being an $\clf^n_t\doteq\sigma\{Y_n(s):s\leq t\}$-stopping time) with $\tau_n\leq N$,
\begin{equation*}
\sup_{n\geq n_0}\sup_{0\leq\theta\leq\del}\PP\{d(Y_n(\tau_n+\theta),Y_n(\tau_n))\geq\eta\}\leq \eps,
\end{equation*}
where $d(\cdot,\cdot)$ is the distance on $\SSS$.
\end{enumerate}
Then $\{Y_n\}_{n\in\NN}$ is tight in $\DD([0,T]:\SSS)$.
\end{theorem}

\subsection{Hilbert-Schmidt and Trace Class Operators}\label{sec:HSInfo}
We collect here some elementary facts about trace class and Hilbert-Schmidt operators. We refer the reader to \cite{reed1980functional} for details. For a separable Hilbert space $\HH$ (with inner product $\lan\cdot,\cdot\ran$ and norm $\|\cdot\|$), let $\cll(\HH)$ be the collection of all bounded linear operators on $\HH$. An operator $A\in\cll(\HH)$ is called  non-negative if $\lan u, Au\ran\geq 0$ for all $u\in\HH$. Such an operator is called trace class if for some CONS $\{e_i\}$ in $\HH$, $\sum_i\lan Ae_i,e_i\ran<\iy$ in which case the quantity is finite (and is the same) for every CONS $\{e_i\}$. An operator $A\in\cll(\HH)$ is called Hilbert-Schmidt if there exists a CONS $\{e_i\}$ in $\HH$ such that $\sum_j\lan Ae_j,Ae_j\ran=\sum_j\|Ae_j\|^2<\iy$. In that case, this quantity is the same for all CONS $\{e_i\}$ and its square root is called the Hilbert-Schmidt norm of $A$, denoted as $\|A\|_{\hs}$. For a non-negative operator $A\in\cll(\HH)$, there is a unique non-negative $B\in\cll(\HH)$ referred to as the non-negative square root of $A$ such that $B^2=A$. If $A$ is a trace class operator, then $B$ is a Hilbert-Schmidt operator.

\subsection{Cylindrical Brownian Motion}\label{sec:CylWeiner}
A collection of continuous real stochastic processes $\{(W_t(h))_{0\leq t\leq T}:h\in\bfell_2\}$ given on a filtered probability space $(\Om,\clf,\PP,\{\clf_t\})$ is called a $\bfell_2$-cylindrical Brownian motion if for every $h\in\bfell_2$, $(W_t(h))_{0\leq t\leq T}$ is a $\{\clf_t\}$-Brownian motion with variance $\|h\|^2_2$ and for $h,k\in\bfell_2$
\begin{equation*}
\lan W(h),W(k)\ran_t = \lan h,k\ran_2t,\ 0\leq t\leq T.
\end{equation*}
For a measurable map $a$ from $[0,T]$ to the space of Hilbert-Schmidt operators from $\bfell_2$ to $\bfell_2$ such that $\int_0^T\|a(s)\|^2_{\hs}ds<\iy$, we denote by $\int_0^ta(s)dW(s)$ the $\bfell_2$-valued martingale defined as the limit of 
\begin{equation*}
\sum_{i=1}^n\sum_{j=1}^n\phi_i\int_0^t\lan\phi_i,a(s)\phi_j\ran_2dW_s(\phi_j)
\end{equation*}
as $n\to\iy$ where $\{\phi_i\}_{i\in\NN}$ is a complete orthonormal system (CONS) in $\bfell_2$. For the fact that the limit exists and is independent of the choice of CONS, we refer the reader to Chapter 4 of \cite{da2014stochastic}.

\bigskip
\noindent {\bf Acknowledgements.}  Research supported in part by the National Science Foundation (DMS-1016441, DMS-1305120), the Army Research Office (W911NF-14-1-0331) and DARPA (W911NF-15-2-0122).

% \bibliographystyle{plain}
% \bibliography{references}

{\sc
\bigskip
\noindent
A. Budhiraja and E. Friedlander\\
Department of Statistics and Operations Research\\
University of North Carolina\\
Chapel Hill, NC 27599, USA\\
email: budhiraj@email.unc.edu, ebf2@live.unc.edu

}
\end{document}